\documentclass[11pt,a4paper,reqno]{amsart}
\usepackage[applemac]{inputenc}
\usepackage{amsmath}
\usepackage{amsthm}
\usepackage{amssymb}
\usepackage{curves}
\usepackage{epic}
\usepackage[all]{xy}
\usepackage{graphicx}
\usepackage{euscript}
\usepackage[colorlinks=true,linktocpage=true,pagebackref=false, urlcolor=black, citecolor=black,linkcolor=black]{hyperref}

\newtheorem{thm}{Theorem}[section]
\newtheorem{prop}[thm]{Proposition}

\newtheorem{cor}[thm]{Corollary}
\newtheorem{lem}[thm]{Lemma}
\newtheorem{fact}[thm]{Fact}

\newtheorem{factsub}{Fact}[subsection]
\newtheorem{propsub}[factsub]{Proposition}
\newtheorem{lemsub}[factsub]{Lemma}
\newtheorem{corsub}[factsub]{Corollary}

\theoremstyle{definition}
\newtheorem{defn}[thm]{Definition}

\newtheorem{remark}[thm]{Remark}
\newtheorem{remarks}[thm]{Remarks}

\newtheorem{remarksub}[factsub]{Remark}
\newtheorem{defnsub}[factsub]{Definition}

\newcommand{\fin}{\hspace*{\fill}$\square$\\ }

\newenvironment{stepsp}{%
\vspace{2mm}\refstepcounter{thm}\noindent{\bf (\thethm)}\ }%
{\em}

\newenvironment{stepsub}{%
\vspace{2mm}\refstepcounter{factsub}\noindent{\bf (\thefactsub)}\ }%
{\em}

\newcounter{substep}
\def\thesubstep{\arabic{substep}}

\newenvironment{substeps}[1]{%
\refstepcounter{substep}\noindent{(\ref{#1}.\thesubstep)\ }\ }%
{\em}

{\em}

\newcommand{\Z}{{\mathbb Z}} \newcommand{\R}{{\mathbb R}}

\newcommand{\sph}{{\mathbb S}}

\newcommand{\sT}{{\EuScript T}}
\newcommand{\bs}{{\EuScript B}}
\newcommand{\bes}{{\EuScript E}}

\newcommand{\sS}{{\EuScript S}}
\newcommand{\sC}{{\EuScript C}}
\newcommand{\sD}{{\EuScript D}}
\newcommand{\sn}{{\EuScript N}}
\newcommand{\sff}{{\EuScript F}}

\newcommand{\so}{{\EuScript O}}
\newcommand{\sj}{{\EuScript J}}
\newcommand{\si}{{\EuScript I}}

\newcommand{\bl}{{\EuScript L}}

\newcommand{\Int}{\operatorname{Int}} \newcommand{\dist}{\operatorname{dist}}

 \newcommand{\Id}{\operatorname{Id}}

 \newcommand{\mult}{\operatorname{mult}}

\newcommand{\x}{{\tt x}}

\newcommand{\veps}{\varepsilon}

\newcommand{\ol}{\overline}

\makeatletter
\renewcommand{\@seccntformat}[1]{\csname the#1\endcsname.\quad}
\makeatother

\numberwithin{equation}{section}

\begin{document}

\bigskip

\title{APPROXIMATION ON NASH SETS WITH MONOMIAL SINGULARITIES}

\author{E. Baro}
\author{Jos\'e F. Fernando}
\address{Departamento de \'Algebra, Facultad de Ciencias Matem\'aticas, Universidad Complutense de Ma\-drid, 28040 MADRID (SPAIN)}
\curraddr{}
\email{eliasbaro@pdi.ucm.es,josefer@mat.ucm.es}
\author{Jes\'us M. Ruiz}
\address{Departamento de Geometr\'\i a y Topolog\'\i a, Facultad de Ciencias Matem\'aticas, Universidad Complutense de Madrid, 28040 MADRID (SPAIN)}
\curraddr{}
\email{jesusr@mat.ucm.es}

\thanks{Authors supported by Spanish GAAR MTM2011-22435}


\date{\today}

\begin{abstract}
This paper is devoted to the approximation of differentiable semialgebraic functions by Nash functions. Approximation by Nash functions is known for semialgebraic functions defined on an affine Nash manifold $M$, and here we extend it to functions defined on \em Nash sets $X\subset M$ whose singularities are monomial\em. To that end we discuss first \em finiteness \em and \em weak normality \em for such sets $X$. Namely, we prove that (i) \em $X$ is the union of finitely many open subsets, each Nash diffeomorphic to a finite union of coordinate linear varieties of an affine space, \em and (ii) \em every function on $X$ which is Nash on every irreducible component of $X$ extends to a Nash function on $M$\em. Then we can obtain approximation for semialgebraic functions and even for certain semialgebraic maps on Nash sets with monomial singularities. As a nice consequence we show that $m$-dimensional affine Nash manifolds with \em divisorial corners \em which are class $k$ semialgebraically diffeomorphic, for $k>m^2$, are also Nash diffeomorphic.
\end{abstract}
\maketitle

\section{Introduction}

The importance of approximation in the Nash setting arises from the fact that this category is highly rigid to work with; for instance, it does not admit partitions of unity, a usual and fruitful tool when available. On the other hand, there exist finite differentiable semialgebraic partitions of unity and therefore approximation becomes interesting as a bridge between the differentiable semialgebraic category and the Nash one. There are already relevant results concerning absolute approximation in the Nash setting (e.g. Effroymson's approximation theorem \cite[\S1]{ef}). An even more powerful tool is relative approximation that allows approximation having a stronger control over certain subsets. In the '80s Shiota developed a thorough study of Nash manifolds and Nash sets (see \cite{s} for the full collected work); among other things, he devised approximation on an affine Nash manifold \em relative to a Nash submanifold\em. Our purpose is to generalize this type of results developing Nash approximation on an affine Nash manifold relative to a Nash subset; of course, as one can expect we need some conditions concerning the singularities on the Nash subset and we will focus on those whose singularities are of \em monomial type\em. This will require a careful preliminary study of this type of singularities in the Nash context. As an interesting application of our results, we prove that the Nash classification of affine $m$-dimensional Nash manifolds with (divisorial) corners is equivalent to the $\sC^k$ semialgebraic classification for $k>m^2$.

Recall that a set $X\subset \R^n$ is \em semialgebraic \em if it is a Boolean combination of sets defined by polynomial equations and inequalities. A semialgebraic set $M\subset \R^n$ is called an \em (affine) Nash manifold \em if it is moreover a smooth submanifold of (an open subset of) $\R^n$; as in this paper all manifolds are affine we often drop the word affine. A function $f:U\to \R$ on an open semialgebraic subset $U\subset M$ is \em Nash \em if it is smooth and semialgebraic, i.e., its graph is semialgebraic. We denote by $\sn(U)$ the ring of Nash functions on $U$. Similarly, we define the ring $\sS^\nu(U)$ of $\sC^\nu$ semialgebraic functions on $U$ for $\nu \geq 1$. The orthogonal projection of a Nash manifold $M\subset \R^n$ into the tangent space at one of its points is a local diffeomorphism at the point; as it is a semialgebraic map we see that a semialgebraic subset $M\subset\R^n$ is an affine Nash manifold of dimension $m$ if and only if every point $x\in M$ has an open neighborhood $U$ in $\R^n$ equipped with a Nash diffeomorphism $(u_1,\ldots,u_n):U\to\R^n$ that maps $x$ to the origin and such that $U\cap M=\{u_{m+1}=\cdots=u_n=0\}$. Even more, the Nash manifold $M$ can actually be covered with \em finitely many \em open sets of this type (see \cite[2.2]{fgr} and \cite[I.3.9]{s}). This kind of \em finiteness \em property permeates all the theory of semialgebraic sets and functions and plays a relevant role in the present paper.

We are interested in Nash sets, that is, the zero sets of Nash functions $f:M\to \R$. It is well known that the set ${\tt Smooth}(Z)$ of smooth points of a semialgebraic set $Z$, that is, the points $x\in Z$ where the germ $Z_x$ equals the germ of an affine Nash submanifold, is a dense open semialgebraic subset of $Z$ (see \cite{st}). We will focus on Nash sets whose singular points (i.e., points which are not smooth) are of a specific form.

\begin{defn}\label{def:monomial}
Let $X\subset M$ be a set and let $x\in X$. The germ $X_x$ is a \em monomial singularity \em if there is a neighborhood $U$ of $x$ in $M$ equipped with a Nash diffeomorphism $u:U\to\R^m$ with $u(x)=0$ that maps $X\cap U$ onto a union of coordinate linear varieties. That is, there is a (finite) family $\varLambda$ of subsets of indices $\lambda=\{\ell_1,\cdots,\ell_r\}$ of possibly different cardinality $r\leq m$  such that
$$
X\cap U={\bigcup}_{\lambda\in\varLambda}\,\{u_{\lambda}=0\}
$$
where $u=(u_1,\ldots,u_m)$ and $\{u_{\lambda}=0\}$ denotes $\{u_{\ell_1}=\cdots=u_{\ell_r}=0\}$. For simplicity we assume that there are no immersed components, that is, if $\lambda,\lambda'\in\varLambda$ are different then $\lambda\nsubseteq\lambda'$ and $\lambda'\nsubseteq\lambda$. This assures that the germs $\{u_\lambda=0\}_x$, $\lambda\in\varLambda$, are the irreducible components of the germ $X_x$. We call $\varLambda$ a \em type \em of the monomial singularity $X_x$; we also say that $X$ has a \em monomial singularity of type $\varLambda$ at $x$\em.
\end{defn}

Of course, a different Nash diffeomorphism $u'$ may provide a different type $\varLambda'$ and therefore a monomial singularity has several types. Thus, two types $\varLambda$ and $\varLambda'$ will be called \em equivalent \em if the union of the coordinate linear varieties given by $\varLambda$ is Nash diffeomorphic to the one given by $\varLambda'$ as germs in the origin. In Section \ref{sec:mono} we show that this equivalence is in fact global via linear isomorphisms and the resulting classes can be characterized arithmetically (Proposition \ref{prop:loadiso}).

Nash monomial singularity germs are a generalization of Nash normal crossings germs, that is,  locally Nash diffeomorphic to a finite union of coordinate linear hyperplanes (see also Section \ref{sec:mono}). Nash normal crossings have been studied extensively due to their relevance in desingularization problems. In \cite{fgr} several finiteness properties of Nash normal crossings were established and our first objective in Sections \ref{sec:mono} to \ref{sec:Nashgen} is to extend them to the much more general setting here. Specifically, in Section \ref{sec:mono} we study in depth the concept of type and in Section \ref{sec:Nashgen} we prove the following semialgebraic property of the monomial singular locus (see \ref{subsec:loci}).

\begin{prop}\label{satype}
Let $X\subset M$ be a Nash set. Fix a type $\varLambda$ and put
$$
T^{(\varLambda)}=\{x\in X: \text{$X$ has a monomial singularity of type $\varLambda$ at $x$}\}.
$$
Then, the set $T^{(\varLambda)}$ is semialgebraic.
\end{prop}

In Section \ref{sec:Nashgen} we analyze the sets of interest in our discussion.

\begin{defn}\label{def:gencrosset}
A Nash set $X\subset M$ has \em monomial singularities in $M$ \em if all germs $X_x$, $x\in X$, are monomial singularities. 
\end{defn}

In fact, we will see that any closed semialgebraic set whose germs are all monomial singularities is a coherent Nash set with monomial singularities (see Lemma \ref{lem:irreducohe}). Now, the main purpose in Section \ref{sec:Nashgen} is the following finiteness property (see the proof in \ref{pfthmfiniteness}).

\begin{thm}\label{thmfiniteness}Let $X\subset M$ be a Nash set with monomial singularities. Then $X$ can be covered with finitely many open sets $U$ of $M$ each one equipped with a Nash diffeomorphism $u:U\rightarrow \R^m$ that maps $X\cap U$ onto a union of coordinate linear varieties.
\end{thm}

To prove the above result we will study first Nash functions on Nash sets with monomial singularities.

\begin{defn}\label{def:nashfunc}
Let $Z$ be a semialgebraic subset of $M$. We say that a function $f:Z\to \R$ is a \em Nash function \em if there exists an open semialgebraic neighborhood $U$ of $Z$ and a Nash extension $F:U\to \R$ of $f$. We denote by $\sn (Z)$ the ring of Nash functions of $Z$. If $X$ is a Nash set then we say that a function $f:X\to \R$ is a  \em ${\tt c}$-Nash or $^{\tt c}\sn$ function \em if its restriction to each irreducible component is a Nash function. We denote by ${}^{\tt c} \sn (X)$ the ring of $^{\tt c}\sn$ functions. Similarly, we define the ring $\sS^\nu(X)$ of $\sC^\nu$ semialgebraic functions on $X$ and the ring ${}^{\tt c}\sS^\nu(X)$ of $^{\tt c}\sS^\nu$ semialgebraic functions on $X$ for $\nu \geq 1$ ($^{\tt c}\sn$ or $^{\tt c}S^\nu$ functions on $X$ are clearly continuous).
\end{defn}

We point out that if $Z$ is a semialgebraic subset of $M$ and $f:Z\to \R$ is a function with a Nash extension to a non necessarily semialgebraic open neighborhood of $Z$ (see the definition in \ref{subsec:basics}) then there is a Nash extension to a semialgebraic one \cite[1.3]{fgr}.
 
We will show in \ref{pfcnashgerms} that by means of analytic coherence of Nash sets with monomial singularities we have the following notorious \em weak normality property \em:

\begin{thm}\label{cnashgerms}
If $X$ is a Nash set with monomial singularities then $\sn(X)= {}^{\tt c}\sn(X)$. 
\end{thm}

The ring $\sn(X)$ has been revealed crucial to develop a satisfactory theory of irreducibility and irreducible components for the semialgebraic setting \cite{fg}; as one can expect such theory extends Nash irreducibility and can be based as well on the ring of analytic functions on the semialgebraic set.

Once we have presented these properties of Nash sets with monomial singularities we approach in Sections \ref{sec:extension}, \ref{section:Aproxfunctions} and \ref{sec:approxmap} the approximation problems. First the ring $\sS^\nu(M)$ of $\sC^\nu$ semialgebraic functions is equipped with an $\sS^\nu$ Whitney topology via $\sS^\nu$ tangent fields (see \ref{topfunc}). The fact that we have $\sS^\nu$ bump functions as well as finite $\sS^\nu$ partitions of unity makes a crucial difference between $\sS^\nu$ and Nash functions and the existence of these glueing functions justify our interest in approximation. In particular, given a closed semialgebraic set $Z$ of $M$ we can extend any $\sS^\nu$ function on $Z$ to $M$ and therefore $\sS^\nu(Z)$ carries the quotient topology making the restriction map $\rho:\sS^\nu(M)\to \sS^\nu(Z)$ a quotient map. If $Z$ is a closed semialgebraic set and $T$ is a semialgebraic subset of $\R^b$ a map $f=(f_1,\ldots,f_b):Z\to T\subset \R^b$ is $\sS^\nu$ if each component $f_i$ is $\sS^\nu$ and $\sS^\nu(Z,T)$ inherits the topology from the product $\sS^\nu(Z,\R^b)=\sS^\nu(Z) \times \cdots \times \sS^\nu(Z)$.

Obviously, approximation problems focus mainly on maps $f:Z\to T$ where either $Z$ is not compact or $T$ is not a Nash manifold. For, by the Stone-Weierstrass approximation theorem \cite[1.6.2]{n2}, if $Z$ is compact any $\sC^\nu$ map $f:Z\to T\subset \R^b$ can be $\sC^\nu$ approximated by a polynomial map $g:Z\to \R^b$. Therefore, if $T$ is a Nash manifold and $g$ is close enough to $f$ then $g(Z)$ is contained in a tubular neighborhood which retracts onto $T$.

In the '80s, Shiota carried a systematic study of approximation over Nash manifolds and showed that if $Z\subset M$ is a (non necessarily compact) closed semialgebraic set then any function in $\sS^\nu(Z)$ can be approximated by Nash functions (see Proposition \ref{prop:absapproxfun}). Of course, it follows that if $N$ is an affine Nash manifold then using a Nash tubular neighborhood of $N$ any map in $\sS^\nu(Z,N)$ can be also approximated by Nash maps in $\sn(Z,N)$ (see Proposition \ref{aproxM}). However, if the codomain is not an affine Nash manifold then the situation gets more complicated. As we have announced before, our purpose in this paper is to obtain this kind of approximation for $\sS^\nu$ maps between certain Nash sets with monomial singularities. Specifically, we say that a Nash set with monomial singularities $X$ is a \em Nash monomial crossings \em if in addition the irreducible components of $X$ are Nash manifolds. Again this mimics the concept of \em normal crossing divisor \em \cite[1.8]{fgr}. Our main result concerning approximation is the following, which is proved in Section \ref{sec:approxmap} (see \ref{pfapproxgc}). 
\begin{thm}\label{approxgc}
Let $X\subset M$ be a Nash set with monomial singularities and let $Y\subset N$ be a Nash monomial crossings. Let $m=\dim(M)$, $n=\dim(N)$ and $q=m\big(\binom{n}{[n/2]}-1\big)$ where $[n/2]$ denotes the integer part of $n/2$. If $\nu\!\ge\! q$ then every $\sS^\nu$ map $f:X\to Y$ that preserves irreducible components can be $\sS^{\nu-q}$ approximated by Nash maps $g:X\to Y$.
\end{thm}
 
To prove this we need to rework approximation for $\sS^\nu$ functions over Nash set with monomial singularities. Indeed, we start in Section \ref{sec:extension} with a result concerning extension of $\sS^\nu$ functions, which is well-known for Nash submanifolds (see Proposition \ref{topirre}): 

\em If $X\subset M$ is a Nash set with monomial singularities, then there are continuous extension linear maps $\theta:\sS^\nu(X)\to\sS^\nu(M)$ such that $\theta(h)|X=h$\em. 

This is stronger than the mere existence of extensions. This extension result will be used to show in Section \ref{section:Aproxfunctions} that Nash functions on Nash sets with monomial singularities can be relatively approximated (and not only absolutely, Proposition \ref{topirre}). Namely: 

\em An $\sS^\nu$ function $F:M\to \R$ whose restriction to $X$ is Nash can be $\sS^{\nu-m}$ approximated by Nash functions $H:M\to \R$ that coincide with $F$ on $X$\em.

Section \ref{sec:corners} is devoted to the classification of affine Nash manifolds with corners, and our approximation results will be crucial to compare $\sS^\nu$ and Nash classifications. This somehow complements Shiota's results on $\sC^\nu$ classification of Nash manifolds \cite[VI.2.2]{s}.

An \em (affine) Nash manifold with corners \em is a semialgebraic set $Q\subset \R^a$ that is a smooth submanifold with corners of (an open subset of) $\R^a$. In \cite{fgr} it is proved that any Nash manifold with corners $Q\subset \R^a$ is a closed semialgebraic subset of a Nash manifold $M\subset \R^a$ of the same dimension and the Nash closure $X$ in $M$ of the boundary $\partial Q$ is a Nash normal crossings.  Note that we can define naturally $\sS^\nu$ functions and maps and their topologies via the closed inclusion of $Q$ in $M$; of course this does not depend on the affine Nash manifold $M$ (see \ref{Nashmancorners} for further details). A Nash manifold with corners $Q$ has \em divisorial corners \em if it is contained in a Nash manifold $M$ as before such that the Nash closure $X$ of $\partial Q$ in $M$ is a normal crossing divisor (as one can expect this is not always the case and a careful study can be found in \cite[1.12]{fgr}).

\begin{thm}\label{difcor}
Let $Q_1$ and $Q_2$ be two $m$-dimensional affine Nash manifolds with divisorial corners. If $Q_1$ and $Q_2$ are $\sS^\nu$ diffeomorphic for some $\nu> m^2$ then they are Nash diffeomorphic. 
\end{thm}

All the basic notions and the notation used in this paper will be explained in the following section. The reader can proceed directly to Section \ref{sec:mono} and refer to the Preliminaries when needed. 

\section{Preliminaries}\label{sfunct}

In this section we introduce the concepts and notation needed in the sequel. First, we adopt the following conventions: $M\subset\R^a$ and $N\subset\R^b$ are affine Nash manifolds of respective dimensions $m$ and $n$; $X$ and $Y$ are Nash subsets of an affine Nash manifold and their irreducible components are denoted by $X_i$ and $Y_j$. The semialgebraic sets will be denoted by $S,T,Z$; on the other hand, $Q\subset\R^a$ is an affine Nash manifold with corners, $\partial Q$ is its boundary and $\Int(Q)=Q\setminus\partial Q$ is its interior. Also $\sS^\nu$ and Nash functions on a semialgebraic set will be denoted by $f,g,h,\ldots$ and their extensions to a larger semialgebraic set by $F,G,H,\ldots$ Homomorphisms between rings of Nash functions, $\sS^\nu$ functions, etc. are denoted using greek letters like $\theta,\gamma,\ldots$ but we reserve as usual $\veps$ and $\delta$ to denote (small) positive real numbers or positive semialgebraic functions involved in approximation of functions.

Let us recall some general properties of semialgebraic sets. Semialgebraic sets are closed under Boolean combinations and by quantifier elimination they are also closed under projections. In other words, any set defined by a first order formula in the language of ordered fields is a semialgebraic set \cite[p.\hspace{2pt}28,\hspace{2pt}29]{bcr}. Keeping this in mind, the basic topological constructions as closures, interiors or borders of semialgebraic sets are again semialgebraic. Also images and preimages of semialgebraic sets by semialgebraic maps are again semialgebraic. The \em dimension \em $\dim(Z)$ of a semialgebraic set $Z$ is the dimension of its algebraic Zariski closure \cite[\S2.8]{bcr}. The \em local dimension $\dim(Z_x)$ of $Z$ at a point $x\in Z$ \em is the dimension $\dim (U)$ of a small enough open semialgebraic neighborhood $U$ of $x$ in $Z$. The dimension of $Z$ coincides with the maximum of those local dimensions. For any fixed $d$ the set of points $x\in Z$ such that $\dim(Z_x)=d$ is a closed semialgebraic subset. 

An important property of compact semialgebraic sets is that they can be triangulated \cite[\S9.2]{bcr}. More relevant for us here is that given a Nash manifold $M\subset \R^a$ and some semialgebraic subsets $Z_1,\ldots,Z_s\subset M$, as a consequence of this triangulability property via a one-point compactification of $M$, there is a \em stratification $\mathcal{G}$ of $M$ compatible with $Z_1,\ldots,Z_s$\em. That is, there exists a finite collection $\mathcal{G}$ of disjoint semialgebraic subsets of $M$ called \em strata \em whose union is $M$ and which are affine Nash manifolds Nash diffeomorphic to some $\R^d$ with the following properties:

(1) The sets $Z_1\ldots,Z_s$ are unions of strata, 

(2) The closure in $M$ of a stratum of $\mathcal G$ is a finite union of strata of ${\mathcal G}$. In particular, one deduces that if $\varSigma,\varGamma$ are strata of ${\mathcal G}$, then either $\varSigma\subset\ol{\varGamma}$ or $\varSigma\cap\ol{\varGamma}=\varnothing$. Therefore, if ${\mathcal G}_{\varSigma}=\{\varGamma\in{\mathcal G}:\ \varSigma\subset\ol{\varGamma}\}$, then the semialgebraic set $\bigcup_{\varGamma\in{\mathcal G}_{\varSigma}}\varGamma$ is an open semialgebraic neighborhood of $\varSigma$ in $M$; this fact will be used freely along the sequel.

(3) Every stratum $\varGamma$ is \em connected at every point $x\in\ol\varGamma$\em, that is, $x$ has a basis of neighborhoods $V$ in $M$ with connected intersection $V\cap\varGamma$. This, together with the fact that $\varGamma$ is an affine Nash manifold, implies that the analytic closure $\ol{\varGamma}_x^{\rm an}$ of the germ $\varGamma_x$ is an irreducible analytic germ of dimension $\dim(\varGamma)$.

(4) Suppose the stratum $\varGamma$ is \em adherent \em to another stratum $\varSigma$, that is, $\varSigma\subset\ol\varGamma$. Then $\varSigma$ has a basis of neighborhoods $V$ in $M$ with connected intersection $V\cap\varGamma$. Following the terminology above, we say that $\varGamma$ is \em connected at $\varSigma$\em.

We refer to \cite[2.3]{fgr} for a more detailed presentation of the triangulation theorem and the deduction of the preceding stratification.

\subsection{Basics on the Nash category.}\label{subsec:basics}The open semialgebraic subsets of $M$ are a base of the topology and therefore Nash functions define a sheaf that we will denote by $\sn$. In particular, the sheaf $\sn$ induces a notion of Nash function $f:U\to \R$ over an arbitrary open subset $U$ of $M$ possibly not semialgebraic. In case $U$ is an open subset of $\R^a$, \em Nash \em means that $f$ is smooth and there exists a nonzero polynomial $P(x,t)\in \R[x,t]=\R[x_1,\ldots,x_a,t]$ such that $P(x,f(x))=0$ for all $x\in U$. If $U$ is semialgebraic this definition coincides with the one in the Introduction. We recall that both $\sn(M)$ and $\sn_x$ are noetherian for any $x\in M$ (see \cite[ 8.2.11, 8.7.18]{bcr}) and therefore it makes sense to consider the Nash closure globally and locally. Moreover, \em Nash functions are analytic and therefore Nash manifolds and Nash sets are analytic sets \em \cite[8.1.8]{bcr}. The ring of germs of analytic functions $\so_x$ is always noetherian even thought the ring of global analytic functions $\so(M)$ is noetherian only in case $M$ is compact. However, any arbitrary intersection of global analytic sets is a global analytic set and so the local and the global analytic closure are well-defined \cite[Ch. V, Prop. 16]{n}. We have:

\begin{factsub}\label{fact:globalananash}Let $Z\subset M$ be a semialgebraic set. Then:

\em(i) \em \cite[2.10]{fg} The Nash and the analytic closures of $Z$ coincide. In particular, if $Z$ is global analytic then it is a Nash set.

\hspace*{-1mm}\em(ii) \em \cite[2.8]{fgr} If $Z$ is a Nash set, then its Nash irreducible components are also its global analytic irreducible components.
\end{factsub}
 
Since $M$ is a semialgebraic subset of $\R^a$ and because of Definition \ref{def:nashfunc} we apparently have two definitions of Nash function $f:M\to \R$. Of course, both notions are equal because \em every affine Nash manifold has a Nash tubular neighborhood\em. 

\begin{factsub}\label{fact:retract}\cite[8.9.5]{bcr} Let $M\subset \R^a$ be a Nash manifold. Then there exists an open semialgebraic neighborhood $U$ of $M$ in $\R^a$ and a Nash retraction $\rho:U\to M$. 
\end{factsub}

\subsection{Coherence and Nash functions.}\label{subsec:strongcoh} 
In the Introduction we have provided the definition of Nash function on a semialgebraic subset of an affine Nash manifold. Now for a Nash set $X$ it will be useful to recover this definition via global sections of a certain sheaf. Of course, the first choice should be the following. Let $\si$ be the sheaf of $\sn$-ideals given by $\si_x=I(X_x)$, the germs of Nash functions vanishing on $X$. The support of $\si$ is $X$ and for any $x\in M$ we have $(\sn/\si)_x=\sn_x/\si_x=\sn(X_x)$. However, as in the analytic setting, the sheaf $\sn/\si$ has in general a bad local-global behaviour and the ring of global sections is not necessarily $\sn(X)$.

Cartan showed that in the analytic context this can be solved using coherence. Let $N$ be an analytic manifold and $\so$ its sheaf of analytic germs. Recall that a sheaf of $\so$-modules $\sff$ is \em coherent \em if (i) $\sff$ is of finite type, i.e., for every $x\in N$ there is an open neighborhood $U$ and a surjective morphism $\so^n|U \rightarrow \sff|U$ and (ii) any morphism $\so^p|U \rightarrow \sff|U$ has a kernel of finite type. By \cite[Prop. 4]{cart} the sheaf $\so$ is coherent and therefore a sheaf of $\so$-ideals is coherent if and only if it is of finite type. The famous Theorems A and B describe the local-global behaviour of coherent sheaves: \em The stalks are spanned by the global sections \em and \em each $p$-cohomology group of the sheaf with $p>0$ is trivial \em \cite[Thm. 3]{cart}. In particular, for any \em coherent locally analytic \em set $Y\subset N$, i.e., locally described by finitely many analytic equations and such that its sheaf $\si_\so$ of $\so$-ideals of 
germs of analytic functions vanishing on $Y$ is coherent, we have the following properties.

\begin{factsub}\cite[Prop. 14, 15]{cart}\label{fact:cohpure} Every coherent locally analytic set $Y$ of an analytic manifold $N$ is (global) analytic. Moreover, if $Y$ is irreducible then it has pure dimension. 
\end{factsub}

Back to the Nash setting, even for sheaves of $\sn$-ideals, coherence is not enough (see the Introduction of \cite{crs3}) because semialgebraicity involves a finiteness phenomenon that mere coherence does not capture. Thus we say that a sheaf $\sff$ of $\sn$-ideals is \em finite \em if there is a finite open semialgebraic covering $\{U_i \}$ of $M$ such that each
$\sff|U_i$ is generated by finitely many Nash functions over $U_i$. Since $\sn$ is coherent as a sheaf of $\sn$-ideals \cite[I.6.6]{s} it follows that a \em finite \em sheaf of $\sn$-ideals is coherent. In \cite{cs} and \cite{crs2} the authors prove that this is the correct notion for sheaves of $\sn$-ideals (we remark that in general for $\sn$-modules the right notion is \em strong coherence \em \cite{crs3}). We suggest the survey \cite{crs1} as a general reference.

\begin{factsub}\label{factfinitesheaf}
Let $\sff$ be a finite sheaf of $\sn$-ideals and $x\in M$. Then \em(A) \em the stalk $\,\sff_x$ is generated by global sections, and \em(B) \em every global section $\sigma$ of $\sn/\sff$ is represented by a global Nash function. 
\end{factsub}

This result assures that if $X$ is a Nash subset of $M$ and $I(X)$ stands for its (finitely generated) ideal $\{f\in \sn(M): f(X)=0\}$, then the sheaf of $\sn$-ideals given by $\sj_x=I(X)\sn_x$ \em is the biggest finite sheaf of ideals whose support is $X$ \em (in fact, as we will see later, is the biggest coherent one).  Moreover, let us see that \em there is a natural isomorphism between the ring $\sn(X)$ and the ring of global sections of $\sn/\sj$\em. 

Indeed, by \cite[I.6.5]{s} we have $\Gamma(U,\sj)=I(X)\sn(U)$ for any open semialgebraic subset $U$ of $M$. Then by Fact \ref{factfinitesheaf}.(B) we deduce $\Gamma(U,\sn/\sj)=\sn(U)/I(X)\sn(U)$. On the other hand, the support of $\sn/\sj$ is $X$ and therefore for any open semialgebraic set $U$ containing $X$ we have $\Gamma(X,\sn/\sj)\equiv\Gamma(U,\sn /\sj)\equiv\Gamma(M,\sn/\sj)$. The fact that $\sn(M)/I(X) \equiv\Gamma(M,\sn /\sj)\equiv \Gamma(U,\sn /\sj) \equiv \sn(U)/I(X)\sn(U)$ means that for every Nash function $g:U\to \R$ there is a Nash function $h:M\to \R$ such that $g-h|_U\in I(X)\sn(U)$ and therefore $g=h$ on $X$. In particular,  \em every function $f\in \sn(X)$ has a Nash extension to $M$\em, so that the restriction homomorphism
\begin{equation}\label{eq:nashsect}
\begin{array}{rcl}
\sn(M) & \to & \sn(X)\\
f & \mapsto & f|_X
\end{array}
\end{equation}
is surjective and hence $\sn(X)=\sn(M)/I(X)\equiv \Gamma(X,\sn/\sj)$, as required.

On the other hand, we also need to describe $^{\tt c}\sn(X)$ as a ring of sections of a sheaf. Let $X_1,\ldots,X_s$ be the irreducible components of $X$ and for every $1\hspace{-2pt}\leq \hspace{-2pt}i,j\hspace{-2pt}\leq \hspace{-2pt} s$ let $\sj_i$ and $\sj_{ij}$ be respectively the sheaves of $\sn$-ideals given by $\sj_{i,x}=I(X_i)\sn_x$ and $\sj_{ij,x}=I(X_i\cap X_j)\sn_x$.
We define the morphism 
$$
\phi: \sn/\sj_1 \times \ldots \times \sn/\sj_s\to \prod_{i<j} \sn/\sj_{ij}
$$
induced by
$$
(f_1,\ldots,f_s) \mapsto  (f_i-f_j)_{i<j}.
$$
Consider the kernel sheaf $\ker(\phi)$. Using equation \eqref{eq:nashsect} we define a \em multiple restriction \em map 
$$
{}^{\tt c}\sn(X)\to \sn(M)/I(X_1)\times \cdots \times \sn(M)/I(X_s):f\mapsto (f|_{X_1},\ldots,f|_{X_s})
$$ 
that provides an isomorphism between ${}^{\tt c}\sn(X)$ and the global sections $$\Gamma(M,\ker(\phi))\subset\sn(M)/I(X_1)\times \cdots \times \sn(M)/I(X_s).$$ Indeed, any global section of $\ker(\phi) \subset \sn/\sj_1 \times \ldots \times \sn/\sj_s$ is represented by global Nash functions $(g_1,\ldots, g_s)$. Since $(g_i-g_j)_x \in I(X_i\cap X_j)\sn_x$ for any $x \in X_i\cap X_j$ we deduce that $g_i=g_j$ on $X_i\cap X_j$ and therefore the function defined by $g(x)=g_i(x)$ if $x\in X_i$ is $\tt c$-Nash on $X$, as required.

Roughly speaking, the coherent sheaf $\sn/ \sj$ captures the global notions concerning Nash functions and $\sn/\si$ the local ones. Hence it is natural to ask when $\si=\sj$ and for the answer we look at the analytic functions. Consider the ideal $I_\so(X)$ of all analytic functions vanishing in $X$. Now, if $X$ is coherent as an analytic set then by Theorem A of Cartan the stalks $\si_{\so,x}$ are generated by the global sections $\Gamma(M,\si_\so)=I_{\so}(X)$. Since $I_{\so}(X)$ is generated by $I(X)$ (see \cite[2.8]{fgr}) it follows that $\sj_x\so_x=I(X)\so_x=I_\so(X)\so_x=\si_{\so,x}\supset \si_x\so_x$ and therefore $\sj_x\so_x=\si_x\so_x$. The monomorphism $\sn_x\hookrightarrow\so_x$ is faithful flat  \cite[8.3.2]{bcr}, so that $\sj_x=\si_x$. Then $\si=\sj$ and therefore $(\sn/\sj)_x=(\sn/\si)_x=\sn_x/\si_x=\sn(X_x)$. In particular, since by definition $\sj$ is coherent as $\sn$-sheaf we have that $X$ is also coherent as a Nash set. Moreover, since $\si_{\so,x}=\si_x\so_x$ \cite[8.6.9]{bcr} it is easy to prove that Nash coherence implies analytic coherence. Summarizing we have:

\begin{propsub}\label{fact:cohananash}
A Nash set $X$ is coherent as an analytic set if and only if it so as a Nash set. 
\end{propsub}

Furthermore, if $X$ is coherent then for any semialgebraic open subset $U$ of $M$ we have
\begin{equation}\label{eq:ix}
I(X)\sn(U)=I(X\cap U)
\end{equation}
where $I(X\cap U)=\{f\in \sn(U): f|_{X\cap U}=0\}$. Indeed, let $\sj'$ be the sheaf of $\sn|_U$-ideals given by $\sj'_x=I(X\cap U)\sn_x$. Since $X\cap U$ is a Nash subset of $U$, we know that $\sn(X\cap U)=\Gamma(X\cap U,\sn|_U/\sj')=\sn(U)/I(X\cap U)$. Moreover, $X\cap U$ is also coherent in $U$ and hence $\sj'=\si|_U$, so that $\sn|_U/\sj'=(\sn/\si)|_U=(\sn/\sj)|_U$. Thus $\sn(X\cap U)=\Gamma(X\cap U,\sn|_U/\sj')=\Gamma(U,\sn/\sj)=\sn(U)/I(X)\sn(U)$.

We finish this discussion with a well-known fact included for the sake of completeness.

\begin{propsub}\label{prop:cohirred}
Let $X$ be a coherent Nash set and let $X_1,\ldots,X_s$ be its Nash irreducible components. Then, for every $x\in X$, each germ $X_{i,x}$ is the union of some Nash irreducible components of the germ $X_x$. In particular, the irreducible components $X_1,\ldots,X_s$ are also coherent.
\end{propsub}
\begin{proof}
By Fact \ref{fact:globalananash} the Nash irreducible components $X_1,\ldots,X_s$ are also the irreducible components in the analytic sense. Since $X$ is a (global) analytic set we can consider its complexification $\widetilde{X}$ whose irreducible components are $\widetilde{X}_1,\ldots, \widetilde{X}_r$ (see \cite[Ch.V, Prop.15,17]{n}). The irreducibility of each $X_i$ implies that each germ $\widetilde{X}_{i,x}$ is pure dimensional \cite[Ch.IV, Cor.4]{n}. The pure dimensionality of the sets $\widetilde{X}_i$ implies, by the Identity Principle, that no irreducible component of the germ $\widetilde{X}_{i,x}$ is contained in $\widetilde{X}_{j,x}$ if $i\neq j$. This shows that the decompositions into irreducible components of $\widetilde{X}_{1,x},\ldots, \widetilde{X}_{s,x}$ provide the decomposition of $\widetilde{X}_x$ into irreducible components. On the other hand, fix a point $x\in M$ and consider the Nash irreducible components $Y_1,\ldots,Y_r$ of the germ $X_x$. These are also the analytic irreducible components \cite[8.3.2]{bcr} and therefore they are coherent \cite[Ch.V, Prop.6]{n}. Furthermore, their complexifications $\widetilde{Y_i}$ are the irreducible components of the complexification $\widetilde{X_x}$ (see \cite[Ch.V, Prop.2]{n}). Now, it follows from the coherence of $X$ that the germ $\widetilde{X}_x$ of the complexification $\widetilde X$ is the complexification $\widetilde{X_x}$ of the germ $X_x$ (see \cite[Prop.12]{cart}). Therefore each $\widetilde{X}_{i,x}$ is the union of some $\widetilde{Y_j}$; hence, each $X_{i,x}$ is the union of some $Y_j$. Thus, $X_{i,x}$ is coherent because it is a finite union of coherent germs \cite[Prop.13]{cart}.
\end{proof}

\subsection{Spaces of differentiable semialgebraic functions.}\label{topfunc} 
Let $M\subset\R^a$ be an affine Nash manifold. Recall that for every integer $\nu\ge 1$ we denote by $\sS^\nu(M)$ the set of all semialgebraic functions $f:M\to\R$ that are differentiable of class $\nu$.

We equip $\sS^\nu(M)$ with the following \em $\sS^\nu$ semialgebraic Whitney topology \em \cite[II.1, p.\hspace{1.5pt}79--80]{s}. Let $\xi_1,\dots,\xi_r$ be semialgebraic $\sS^{\nu-1}$ tangent fields on $M$ that span the tangent bundle of $M$. For every continuous semialgebraic function $\veps:M\to\R^+$ we denote by ${\mathcal U}_\veps$ the set of all functions $g\in\sS^\nu(M)$ such that
$$
|g|<\veps \quad\text{and}\quad|\xi_{i_1}\cdots\xi_{i_\mu}(g)|<\veps\, 
 \text{ for }\, 1\!\le\! i_1,\dots,i_\mu\!\le\! r, 1\!\le\!\mu\!\le\!\nu. 
$$
These sets ${\mathcal U}_\veps$ form a basis of neighborhoods of the zero function for a topology in $\sS^\nu(M)$ that does not depend on the choice of the tangent fields. It is well-known that each $\sS^\nu(M)$ is a Hausdorff topological ring \cite[II.1.6]{s}, but neither a topological vector space nor a Fr\'echet space. Note that the obvious inclusions $\sS^\nu(M)\subset\sS^\mu(M)$, $\nu>\mu$, are continuous. Moreover, since semialgebraic smooth functions on $M$ are Nash by definition, we have $\sn(M)={\bigcap}_\nu\sS^\nu(M)$. The first important result is that the inclusion $\sn(M)\subset \sS^\nu(M)$ is dense.

\begin{factsub}\cite[II.4.1, p.\hspace{1.5pt}123]{s}\label{shiota}
Every semialgebraic $\sS^\nu$ function on $M$ can be approximated in the $\sS^\nu$ topology by Nash functions.
\end{factsub}

It is clear that given a semialgebraic set $Z\subset M$ the zero-ideal $I^\nu(Z)=\{f\in\sS^\nu(M): f|Z\equiv0\}$ is closed. Indeed, if $f(z)\ne0$ for some $z\in Z$, no function vanishing on $Z$ can be closer than the constant semialgebraic function $\veps=|f(z)|>0$ to $f$. The closedness of $I^\nu(Z)$ in $\sS^\nu(M)$ is the standard assumption to equip $\sS^\nu(M)/I^\nu(Z)$ with the quotient topology. Furthermore, with this topology the continuous quotient map $\sS^\nu(M)\to \sS^\nu(M)/I^\nu(Z)$ is also open\label{restop}: if $\mathcal U$ is an open set in $\sS^\nu(M)$, its saturation
$$
{\mathcal U}^Z=\bigcup_{f\in\,\mathcal U}(f+I^\nu(Z))=\bigcup_{h\in I^\nu(Z)}\!\!(\mathcal U+h)
$$
is open as translations are homeomorphism of the ring $\sS^\nu(M)$.

Now, if the semialgebraic set $Z$ is \em closed \em in $M$ we can identify $\sS^\nu(Z)$ with the quotient $\sS^\nu(M)/I^\nu(Z)$. Indeed, every function in $\sS^\nu(Z)$ can be extended to $M$: we can use bump $\sS^\nu$ functions to see that a semialgebraic $\sS^\nu$ function $f:U\to \R$ defined on an open semialgebraic neighborhood $U$ of $Z$ coincides on a perhaps smaller neighborhood with a semialgebraic $\sS^\nu$ function $F:M\to \R$. Therefore, \em we have a topology on $\sS^\nu(Z)$ such that the restriction map $\rho:\sS^\nu(M)\to \sS^\nu(Z)$ is an open quotient homomorphism\em. In general, for any two semialgebraic sets $Z\subset Z'$ the composition of the restrictions $\sS^\nu(M)\to \sS^\nu(Z)$ and $\sS^\nu(Z)\to \sS^\nu(Z')$ coincide with the restriction $\sS^\nu(M)\to \sS^\nu(Z')$, so that:

\begin{stepsub}\label{openq}
 \em Given two closed semialgebraic sets $Z'\subset Z$ the restriction map $\sS^\nu(Z)\to \sS^\nu(Z')$ is an open quotient map. \em
\end{stepsub}

In several situations we have to work with a semialgebraic set $N$ of $M$ which is in turn a closed affine Nash submanifold of $M$. We already pointed out that the intrinsic notion of semialgebraic $\sS^\nu$ map on $N$ as manifold equals the one coming from $M$ as semialgebraic set because of the existence of Nash tubular neighborhoods. Furthermore, the topology in $\sS^\nu(N)$ coming from the tangent fields of $N$ and the one as a closed semialgebraic subset of $M$ are equal:

\begin{propsub}\label{prop:tosubpmanifold}
Let $N\subset M$ be affine Nash manifolds with $N$ closed in $M$. Then the restriction homomorphism $\sS^\nu(M)\to\sS^\nu(N)$ is open when the $\sS^\nu$ topologies are defined through tangent fields.
\end{propsub}
\begin{proof}
Consider $\sS^\nu(M)$ and $\sS^\nu(N)$ endowed with the topologies defined through tangent fields. It is known that then the restriction is continuous \cite[II.1.5, p.\hspace{1.5pt}83]{s}, and we are to see it is open. The key fact is that there is a continuous linear map $\theta:\sS^\nu(N)\to\sS^\nu(M)$ with $\theta(h)|_{N}=h$ (see \cite[II.2.14, p.\hspace{1.5pt}107]{s}). Let $F$ be a semialgebraic $\sS^\nu$ function on $M$ and let $f=F|_N$. We must see that if $g\in\sS^\nu(N)$ is close enough to $f$, then there is a semialgebraic $\sS^\nu$ extension $G$ of $g$ to $M$ arbitrarily close to $F$. But given a $\theta$ as above, the function $G=F+\theta(g)-\theta(f)$ is close to $F$ when $g$ is close to $f$.
\end{proof}

Note that in the preceding result we are using the existence of a continuous extension linear map $\sS^\nu(N)\to \sS^\nu(M)$ when $N$ is a closed Nash submanifold of $M$. This property is stronger than mere extension. In Proposition \ref{topirre} we will prove that such extension linear maps exist for Nash sets with monomial singularities.

\begin{remarksub} 
Observe that all we put so far can be done for $\sS^\nu$ manifolds, that is, semialgebraic sets which are differentiable submanifolds of class $\nu$ (see \cite[I.3.11(i), p.\hspace{1.5pt}30]{s}). Instead of Nash tubular neighborhood we have to use \em bent \em tubular neighborhoods \cite[II.6.1, p.\hspace{1.5pt}135]{s}.
\end{remarksub}

In the case of a Nash set $X\subset M$ it is also natural to consider the ring ${}^{\tt c}\sS^\nu(X)$: if $X_1,\ldots,X_s$ are the irreducible components of $X$ then $f\in {}^{\tt c}\sS^\nu(X)$ means $f|_{X_i}\in \sS^\nu(X_i)$ for $1\leq i \leq s$. The multiple restriction homomorphism ${}^{\tt c}\sS^\nu(X)\to \sS^\nu(X_1)\times \cdots \times \sS^\nu(X_s)$ is injective and therefore \em we consider in ${}^{\tt c}\sS^\nu(X)$ the topology induced by the product topology of $\sS^\nu(X_1)\times \cdots \times \sS^\nu(X_s)$ \em. Note that the image of ${}^{\tt c}\sS^\nu(X)$ in $\sS^\nu(X_1)\times \cdots \times \sS^\nu(X_s)$ is closed. Indeed, this image is the kernel of the homomorphism $$\sS^\nu(X_1)\times \cdots \times \sS^\nu(X_s)\to \prod_{i<j} \sS^\nu(X_i\cap X_j):(f_1,\ldots,f_s)\mapsto (f_i|_{X_i\cap X_j}-f_j|_{X_i\cap X_j})_{i<j}.$$ This homomorphism is continuous because $\sS^\nu(M)$ is a topological ring and therefore its kernel is closed.

Moreover, the inclusion
$$
\gamma:\sS^\nu(X)\hookrightarrow {}^{\tt c}\sS^\nu(X):f\mapsto f
$$
is continuous, but one cannot expect it to be surjective in general. Since the embedded image of ${}^{\tt c}\sS^\nu(X)$ in $\sS^\nu(X_1)\times \cdots \times \sS^\nu(X_s)$ is closed, if $\gamma$ is a homeomorphism then the multiple restriction homomorphism $\sS^\nu(X)\to \sS^\nu(X_1)\times \cdots \times \sS^\nu(X_s)$ is a closed embedding. As we will see in Proposition \ref{topirre} this happens whenever $X$ is a Nash set with monomial singularities.

Finally, we point out that Shiota's approximation for affine Nash manifolds implies readily a rather general absolute approximation result. Indeed, since the restriction homomorphism $\sS^\nu(M)\to\sS^\nu(Z)$ is onto and continuous, from Fact \ref{shiota} we obtain:

\begin{propsub}\label{prop:absapproxfun}
If $Z$ is a closed semialgebraic set of $M$ then every $\sS^\nu$ function $f:Z\to\R$ can be $\sS^\nu$ approximated by Nash functions. 
\end{propsub}

\subsection{Spaces of differentiable semialgebraic maps.}\label{smapp}
Let $M\subset\R^a$ be a Nash manifold. Let $Z\subset M$ and $T\subset\R^b$ be semialgebraic sets. 
A semialgebraic map 
$f:Z\to T$ is $\sS^\nu$ when so are its components: we write $f=(f_1,\dots,f_b):Z\to T\subset\R^b$ and then check if the functions $f_k:Z\to\R$ are $\sS^\nu$. We denote $\sS^\nu(Z,T)$ the set of all $\sS^\nu$ maps $Z\to T$. We suppose henceforth that $Z$ is \em closed \em in $M$ in order to endow $\sS^\nu(Z,T)$ with a topology. We use the canonical inclusions
$$
\sS^\nu(Z,T)\subset\sS^\nu(Z,\R^b)=\sS^\nu(Z,\R)\times\overset{b}{\cdots}\times\sS^\nu(Z,\R):f\mapsto(f_1,\dots,f_b).
$$
The product has of course the product topology, and then $\sS^\nu(Z,T)$ is equipped with the subspace topology. Roughly speaking, $g$ is close to $f$ when its components $g_k$ are close to the components $f_k$ of $f$. This definition is the one used in the case of semialgebraic $\sS^\nu$ manifolds \cite[II.1.3, p.\hspace{1.5pt}80]{s}. If $T$ is contained in another semialgebraic subset $T'$, the inclusion $\sS^\nu(Z,T)\subset\sS^\nu(Z,T')$ is an embedding. If $X$ is a Nash set, we say again that a semialgebraic map $f=(f_1,\ldots,f_b):X\to T$ is ${}^{\tt c}\sS^\nu$ when so are its components $f_k$ and we denote by ${}^{\tt c}\sS^\nu(X,T)$ the set of all such maps. If $X_1,\ldots, X_s$ are the irreducible components of $X$ then the multiple restriction homomorphism ${}^{\tt c}\sS^\nu(X,T)\to \sS^\nu(X_1,T)\times \cdots \times \sS^\nu(X_s,T)$ is injective and therefore we consider in ${}^{\tt c}\sS^\nu(X,T)$ the topology induced by the product topology of $\sS^\nu(X_1,T)\times \cdots \times \sS^\nu(X_s,T)$. Again, the image of ${}^{\tt c}\sS^\nu(
X,
T)$ in $\sS^\nu(X_1,T)\times \cdots \times \sS^\nu(X_s,T)$ is closed.

Some basics on spaces of functions pass immediately to spaces of maps. For instance, given two closed semialgebraic sets $Z'\subset Z\subset M$, the restriction map
$$
\sS^\nu(Z,T)\to\sS^\nu(Z',T):f\mapsto f'= f|_{Z'}
$$
is continuous. On the other hand, we know that if $T=\R$ this restriction is an open quotient, from which the same follows for $T$ open in $\R^b$ (as then $\sS^\nu(Z,T)$ is open in $\sS^\nu(Z,\R^b)$). It is well-known that already for topological reasons surjectivity fails in general: for instance, the identity $\sph^1\to\sph^1$ has no continuous extension $\R^2\to\sph^1$. We want to study in which situations the restriction map is at least open. A useful fact will be that composite on the left is continuous.

\begin{propsub}\label{comr}
Let $Z\subset M$ be a closed semialgebraic set and let $T\subset \R^b$ be a locally compact semialgebraic set. Let $h:T\to T'\subset\R^c$ be an $\sS^\nu$ map of semialgebraic sets. Then the map
$$
h_*:\sS^\nu(Z,T)\to\sS^\nu(Z,T'):f\mapsto h\circ f
$$
is continuous for the $\sS^\nu$ topologies. 
\end{propsub}
\begin{proof}
Since $T$ is locally compact, it is closed in some open semialgebraic set $U\subset\R^b$, and there is a semialgebraic $\sS^\nu$ map $H:U\to\R^c$ that extends $h$. We have the following commutative diagram:
\setlength{\unitlength}{1mm}
\begin{center}
\begin{picture}(25,26)(0,-2)
\put(12,22){\makebox(0,0){\,$\sS^\nu(Z,T)\overset{\,h_*}{\longrightarrow}\,\sS^\nu(Z,T')$}}
\put(0,18){\vector(0,-1){4.5}}\put(24,18){\vector(0,-1){4.5}}
\put(-2.8,15){\footnotesize$j$}\put(25.5,15){\footnotesize$j'$}
\put(12,11){\makebox(0,0){$\sS^\nu(Z,U)\overset{\,H_*}{\longrightarrow}\sS^\nu(Z,\R^c)$}}
\put(0,3){\vector(0,1){4.5}}\put(24,3){\vector(0,1){4.5}}
\put(-3,4.5){\footnotesize$\rho$}\put(25.5,4.5){\footnotesize$\rho$}
\put(12,0){\makebox(0,0){$\sS^\nu(M,U)\overset{\,H_*}{\longrightarrow}\sS^\nu(M,\R^c)$}}
\end{picture}
\end{center}
where $j$ and $j'$ stand for the canonical embeddings and $\rho$ for the restriction quotients.

First we explore the lower square. Here we know that the lower $H_*$ is continuous (the Nash manifolds case \cite[II.1.5, p.\hspace{1.5pt}83]{s}), hence the composition $\rho\circ H_*$ is continuous too. The latter map coincides with $H_*\circ\rho$, which is thus continuous. But $\rho$ is a quotient map (the target is open in an affine space), hence the middle $H_*$ is continuous. Now we turn to the upper square. As we have just seen that $H_*$ is continuous, the composite $H_*\circ j$ is continuous. But this map coincides with $j'\circ h_*$, which is consequently continuous. As $j'$ is an embedding, $h_*$ is continuous.
\end{proof}

Back to restrictions we deduce the following (the same result is true if we have a $\sS^\nu$ manifold instead of a Nash manifold).

\begin{propsub}
Let $Z'\subset Z$ be two closed semialgebraic sets, and let $N\subset \R^b$ be a Nash manifold. Then the restriction $\sS^\nu(Z,N)\to\sS^\nu(Z',N):f\mapsto f|_{Z'}$ is an open map.
\end{propsub}
\begin{proof}
Let $U$ be an open semialgebraic tubular neighborhood of $N$ equipped with a Nash retraction $\eta:U\to N$ (that is $\eta(y)=y$ for $y\in N$). We must see that for any open neighborhood $\mathcal V$ of a function $f\in\sS^\nu(Z,N)$, every close enough approximation $g'\in \sS^\nu(Z',N)$ of $f|Z'$ has an extension $g\in\mathcal V$. By definition, $g'$ is close to $f|Z'$ if and only if the components $g'_k$ of $g'$ are close to those $f_k|Z'$ of $f|Z'$. But since the restriction of functions is open, if $g'_k$ is close enough to $f_k|Z'$, then it has an extension $g_k$ to $Z$ close to $f_k$. Thus the $g_k$'s are the components of an extension $g:Z\to\R^b$ of $g'$ which is close to $f$, and we can take it close enough for its image to be contained in $U$. To obtain a map into $N$ we take $h=\eta\circ g: Z\to N$. This settles the matter because composition on the left is continuous (Proposition \ref{comr}), and so, $h(=\eta\circ g)$ is close to $f(=\eta\circ f)$, as wanted.
\end{proof} 

Nash tubular neighborhoods also help to get absolute approximation for \em maps \em into Nash manifolds.

\begin{propsub}\label{aproxM}
Let $Z\subset M$ be a closed semialgebraic set and let $N\subset\R^b$ be a Nash manifold. Every $\sS^\nu$ map $f:Z\to N$ can be $\sS^\nu$ approximated by Nash maps.
\end{propsub}
\begin{proof}Pick first a Nash retraction $\eta:W\to N$ defined on an open semialgebraic tubular neighborhood $W\subset\R^b$ of $N$. Let us write $f=(f_1,\ldots,f_b):Z\to N\subset \R^b$. By Proposition \ref{prop:absapproxfun} every component $f_k:Z\to \R$ can be $\sS^\nu$ approximated by a Nash function $g_k:Z\to \R$. Thus the $\sS^\nu$ map $g=(g_1,\ldots,g_k):Z\to \R^b$ approximates $f$, and for a closed enough approximation, we have $g(Z)\subset W$. By Proposition \ref{comr}, $\eta \circ g$ approximates $f(=\eta \circ f)$.
\end{proof}

\subsection{Nash manifolds with corners.}\label{Nashmancorners}
As in the case of Nash manifolds, a semialgebraic set $Q\subset \R^a$ is an $m$-dimensional Nash manifold with corners if every point $x\in Q$ has an open neighborhood $U$ of $\R^a$ equipped with a Nash diffeomorphism $(u_1,\ldots, u_n):U\to \R^a$ that maps $x$ to the origin and with $U\cap Q=\{ u_1\geq 0,\ldots,u_r\geq 0, u_{m+1}=0,\ldots,u_n=0\}$ for some $0\leq r \leq m$. In \cite{fgr} several properties of affine Nash manifolds with corners are established: for instance, and as in the case of Nash manifolds, that $Q$ can be covered with finitely many open sets of this type. The most relevant property for our purposes here is the following. Recall that the boundary of $Q$ is $\partial Q=Q\setminus{\tt Smooth}(Q)$, where ${\tt Smooth}(Q)$ is the set of smooth points of $Q$.
 
\begin{factsub}\cite[1.12, 6.5]{fgr}\label{envelope}
Let $Q\subset \R^a$ be a Nash manifold of dimension $m$ with corners. Then there exists a Nash manifold $M\subset \R^a$ of the same dimension $m$ such that $Q$ is a closed semialgebraic subset of $M$ and the boundary $\partial Q$ is a closed semialgebraic set whose Nash closure $X$ in $M$ is a normal crossings at every point and satisfies $X\cap Q=\partial Q$.
\end{factsub}

We will call such an $M$ a \em Nash envelope \em of $Q$. Note that the boundary $\partial Q$ defined above coincides with the topological boundary of $Q$ in $M$ and that the interior as manifold $\Int(Q)=Q\setminus \partial Q$ coincides with the topological interior of $Q$ in $M$. Moreover, if $M'\subset M$ is an open subset that contains $Q$, then $M'$ is also a Nash envelope of $Q$. This is the usual procedure to obtain a Nash envelope with additional properties: to drop the closed semialgebraic set $C=M\setminus M'$. Thus, replacing $M$ by $M'$ or \em making $M$ smaller \em amounts to drop a closed semialgebraic set $C\subset M$ disjoint from $Q$. All in all, we see that the concept of Nash envelope works as a germ at $Q$. 

The proof of the preceding fact uses the following result, which we will also need in the sequel:
\begin{factsub}\label{closedinmani}\cite[1.2]{fgr}  Let $Z\subset \R^a$ be a locally compact semialgebraic set such that for each $x\in Z$ the analytic closure $\overline{Z}_x^{\text{an}}$ of the germ $Z_x$ is smooth of constant dimension $m$. Then $Z$ is a closed subset of a Nash manifold $M\subset \R^a$ of dimension $m$.
\end{factsub}

Now we want to define $\sS^\nu$ functions for affine Nash manifolds with corners. Since an affine Nash manifold with corners $Q\subset \R^a$ is locally compact, $Q$ is closed in the open semialgebraic subset $V=\R^a\setminus (\ol{Q} \setminus Q)$ of $\R^a$. Consequently, \em we define $\sS^\nu(Q)$ and its topology using the closed inclusion of $Q$ in $V$\em, as we did in Section \ref{topfunc}. Let us see that this topology does not depend on the neighborhood (notice that the argument for maps into $\R^b$ is similar). 

Indeed, let $W\subset V$ be another open semialgebraic neighborhood of $Q$ and consider an $\sS^\nu$ partition of unity $\varphi_1,\varphi_2:V\to \R$ subordinated to the covering $\{V\setminus Q,W\}$. Then the maps $\varPsi_{VW}:\sS^\nu(V)\to \sS^\nu(W): f\mapsto \varphi_2 f$ and $\varPsi_{WV}:\sS^\nu(W)\to \sS^\nu(V):g\mapsto \varphi_1+\varphi_2 g$ are continuous; moreover, $\varphi_2f=f$ and $\varphi_1+\varphi_2g=g$ on $Q$. Thus, we have the commuting diagram 
$$\xymatrix{\sS^\nu(V) \ar[d]_{\rho_V} \ar@<0.5ex>[r]^{\varPsi_{VW}} & \sS^\nu(W)\ar@<0.5ex>[l]^{\varPsi_{WV}}\ar[d]^{\rho_W} \\ \sS^\nu(Q) \ar@<0.5ex>[r]^{\text{Id}} & \sS^\nu(Q)\ar@<0.5ex>[l]^{\text{Id}} } $$
where $\rho_V:\sS^\nu(V)\to \sS^\nu(Q) $ and $\rho_W:\sS^\nu(W)\to\sS^\nu(Q)$ are the open quotient homomorphisms. The identity map from left to right is continuous if and only if  $\text{Id}\circ \rho_V$ is continuous if and only if $\rho_W \circ  \varPsi_{VW}$ is continuous, which is true. The continuity of the identity map from right to left is similar and hence the topology does not depend on the neighborhood. 

On the other hand, we could define $\sS^\nu(Q)$ and its topology via the closed inclusion of $Q$ into any Nash envelope $M$. Using a Nash retraction $\rho:V\to M$ of $M$ it follows that both definitions coincide (note that $M$ is closed in $V$ and therefore we can apply Proposition \ref{prop:tosubpmanifold}).

In general, if $Q\subset \R^a$ is a Nash manifold with corners and $M\subset \R^a$ a Nash envelope of $Q$, then the Nash closure $X$ of the boundary $\partial Q$ is not a Nash normal crossing divisor, i.e., its irreducible components need not to be Nash manifolds. In \cite[1.12]{fgr} there is a full characterization of Nash manifolds with corners for which this is true for $M$ small enough: such a $Q$ is called here a \em Nash manifold with divisorial corners\em. 

To progress further we introduce the following notion: a \em face \em of a Nash manifold with corners $Q\subset \R^a$ is the (topological) closure of a connected component of ${\tt Smooth}(\partial Q)$; of course, $\partial Q$ is the union of all the faces. This notion of faces are used to characterize divisorial corners. We state that characterization as it is more convenient here:

\begin{factsub}\label{facesmani}\cite[1.12, 6.4, 6.5]{fgr} Let $Q\subset \R^a$ be a Nash manifold with divisorial corners and let $M\subset \R^a$ be a Nash envelope of $Q$. Let $X$ be the Nash closure of $\partial Q$ in $M$. Then, if we make $M$ small enough, we have:

\em (1) \em $X$ is a normal crossing divisor in $M$.

\em (2) \em Let $D_1,\ldots,D_r$ be the distinct faces of $Q$ and let $X_1,\ldots,X_r$ be their Nash closure in $M$. Then 
the distinct irreducible components of $X$ are $X_1,\ldots,X_r$.

\em (3) \em All faces $D_i$ are again Nash manifolds with divisorial corners.

\em (4) \em $X_i$ is a Nash envelope of $D_i$ satisfying $X_i\cap Q=D_i$ for $1\leq i \leq r$.

\em (5) \em The number of faces of $Q$ that contain a given point $x\in \partial Q$ coincides with the number of connected components of the germ ${\tt Smooth}(\partial Q_x)$.
\end{factsub}

Note that (3) and (5) are intrinsic and do not depend on $M$. We finish this section with an iterated construction of faces that fits properly in the framework of this paper.

\subsection{Iterated faces of Nash manifolds with divisorial corners.}\label{facesNashmancorners}
Let $Q$ be a connected Nash manifold with divisorial corners. By Fact \ref{facesmani}(3) the iterated faces are manifolds with divisorial corners. 

We start with one single \em $m$-face \em: $Q$ itself. Then by descending induction, for $d < m$ a \em $d$-face \em is a face of a $(d+1)$-face. Clearly, a $d$-face has dimension $d$. The induction ends at some $d=m_0\geq 0$. In particular, the $m_0$-faces are Nash manifolds. These data are the same for all Nash envelopes $M$.

Next, for a given Nash envelope $M$ of $Q$, we will consider the Nash closures of the iterated faces. Notice how these Nash closures vary when shrinking $M$ to $M'\subset M$ containing $Q$: if $Z$ is the Nash closure in $M$ of a $d$-face $D$ then the Nash closure $Z'$ of $D$ in $M'$ is the irreducible component of $Z \cap M'$ containing $D$.

Moreover, we can apply Fact \ref{facesmani} in each step of the construction and therefore:

\begin{lemsub}\label{envelopefaces}For $M$ small enough we can assume that if $D$ is a $d$-face then its Nash closure $Z$ in $M$ is a Nash envelope of $D$ such that the Nash closure of $\partial D$ is a normal crossing divisor in $Z$ and  $Z\cap Q=D$.
\end{lemsub}
\begin{proof}Indeed, note first that if it holds for some $d$-face $D$ and we shrink $M$ to $M'\supset Q$ then $D$ retains the property. This is a consequence of the way Nash closures vary when we shrink $M$.

Now we argue by descending induction on $d$. For $d=m$ the statement reduces to Fact \ref{facesmani}(1). Now, let $D$ be a $d$-face, $d<m$, which is a face of a $(d+1)$-face $D'$. By induction the Nash closure $Z'$ of $D'$ in $M'$ is a Nash envelope of $D'$ satisfying $Z'\cap Q=D'$. 

We apply again \ref{facesmani}(4) to $D'\subset Z'$ and find a closed semialgebraic set $C'\subset Z'\setminus D'$ such that the Nash closure $Z$ of $D$ in $Z'\setminus C'$ is a Nash envelope of $D$ with $Z\cap D'=D$. Moreover, since $Z'\cap Q=D'$ we have that $C'\cap Q=\varnothing$ and therefore we can replace $M$ by $M'=M\setminus C'$. The Nash closure of $D$ in $M'$ is the same $Z$ and $D\subset Z\cap Q\subset Z\cap Z' \cap Q\subset Z\cap D'=D$, so that $ Z\cap Q=D$. 

Finally, we apply \ref{facesmani}(1) to $D\subset Z$ and we get a closed semialgebraic set $C\subset Z\setminus D$ such that the Nash closure of $\partial D$ in $Z\setminus C$ is a normal crossing divisor. Since $Z\cap Q=D$ we have that $C\cap Q=\varnothing$ and therefore we can replace $M'$ by $M''=M'\setminus C$.
\end{proof}

Let us look at this construction locally, that is, 
$$
M=\R^m \text{ and } Q=\{x_1\ge 0,\ldots,x_{s}\ge 0 \} \text{ with } 0\le s \le m-m_0.
$$
In this case the $d$-faces are $\{x_{i_1}=\ldots=x_{i_{m-d}}=0 \}\cap Q$ for $1\le i_1,\ldots, i_{m-d}\le s$ and the Nash closures of such a face is $\{x_{i_1}=\cdots=x_{i_{m-d}}=0 \}$. \em We see that any intersection of Nash closures of faces is again the Nash closure of a face.\em

Again in the general setting $Q\subset M$, let us fix $x\in Q$. We can make the same construction of iterated faces for the germ $Q_x$, which can be clearly described in the local model. More relevant is the following: 

\begin{lemsub}\label{germfaces}
The distinct (Nash closures of the) $d$-faces of the germ $Q_x$ are the germs at $x$ of the (resp. Nash closures of the) $d$-faces of $Q$.
\end{lemsub}
\begin{proof}We argue by descending induction on $d$. For $d=m$ it is obvious, so suppose $d<m$. By induction the $(d+1)$-faces of $Q_x$ are the germs $D'_x$ of the $(d+1)$-faces $D'$ of $Q$ and consequently the $d$-faces of $Q_x$ are the faces of those $D'_x$. Let $D_1,\ldots,D_s$ be the faces of a fixed $(d+1)$-face $D'$. We claim that the germs $D_{1,x},\ldots, D_{s,x}$ are the faces of the germ $D'_x$. It is clear that each $D_{i,x}$ is a union of faces of $D'_x$, hence what we claim is that $D'_x$ has exactly $s$ faces. This is Fact \ref{facesmani}(5), and we are done.

Finally, let $Z$ be the Nash closure of a $d$-face $D$. The Nash closure $Y_x$ of the germ $D_x$ is included in $Z_x$ and since by \ref{envelopefaces} both are smooth of the same dimension, $Y_x=Z_x$.\end{proof}   
\begin{propsub}\label{facesoffaces}In the setting above, if $M$ is small enough then for any two Nash closures $Z_1$ and $Z_2$ of iterated faces we have that $Z_1\cap Z_2$ is a Nash manifold and any irreducible component of $Z_1\cap Z_2$ meeting $Q$ is again the Nash closure of some iterated face.
\end{propsub}
\begin{proof}Consider Nash closures $Z_1$ and $Z_2$ of iterated faces $D_1$ and $D_2$. By \ref{germfaces}, for every $x\in Q$ we have that $Z_{1,x}$ and $Z_{2,x}$ are Nash closures of iterated faces of $Q_x$ and therefore their intersection $Z_{1,x}\cap Z_{2,x}=(Z_1\cap Z_2)_x$ is the Nash closure of a face of $Q_x$. In particular, for every $x\in Q$ we have that $(Z_1\cap Z_2)_x$ is smooth. Thus, the closed semialgebraic set $C=(Z_1\cap Z_2)\setminus {\tt Smooth}(Z_1\cap Z_2)$ does not meet $Q$ and we shrink $M$ to $M'=M\setminus C$. Let $Z'_1$ and $Z'_2$ be the Nash closures of $D_1$ and $D_2$ in $M'$. For all $x\in Z'_1\cap Z'_2$, the germ $Z'_{1,x}\cap Z'_{2,x}=Z_{1,x}\cap Z_{2,x}=(Z_1\cap Z_2)_x$ is smooth.

Now, let $Y$ be a connected component of $Z_1\cap Z_2$ such that there is $x\in Y\cap Q$. Since $(Z_1\cap Z_2)_x$ is smooth we have that $Y_x=(Z_1\cap Z_2)_x$. As before, 
$Z_{1,x}\cap Z_{2,x}=(Z_1\cap Z_2)_x$ is the Nash closure of a face of $Q_x$ and therefore, by \ref{germfaces}, we have that $(Z_1\cap Z_2)_x=Z_x$ where $Z$ is the Nash closure of an iterated face. In particular, since $Y_x=(Z_1\cap Z_2)_x=Z_x$ and both $Y$ and $Z$ are connected Nash manifolds we deduce $Y=Z$.
\end{proof}

\section{Monomial singularity types}\label{sec:mono}
As said in the Introduction a set $X\subset M$ has a monomial singularity at $x\in X$ if the point has a neighborhood $U$ in $M$ equipped with a Nash diffeomorphism $u=(u_1,\ldots,u_m):U\to\R^m$ such that $u(x)=0$ and 
$$
X\cap U={\bigcup}_{\lambda\in\varLambda}\{u_{\lambda}\!=0\}, \quad \text{where \ } \{u_\lambda=0\}=\{u_{\ell_1}=\cdots=u_{\ell_r}=0\},
$$ 
for a certain type $\varLambda$ (see Definition \ref{def:monomial}), that is, the germ $X_x$ is a monomial singularity of type $\varLambda$. Note that any $u$ as above maps each irreducible component $X_i$ of $X_x$ onto some coordinate linear variety $L_{\lambda(i)}=\{u_{\lambda(i)}=0\}$, and so each $X_i$ is non-singular, as well as any intersection $X_{i_1}\cap\cdots\cap X_{i_p}$. Furthermore, the derivative $d_xu:T_xM\to\R^m$ maps the tangent space $T_xX_i$ onto that of $L_{\lambda(i)}$, which is the same $L_{\lambda(i)}$. Thus, $v=(d_xu)^{-1}\circ u:U\to\R^a$ is a diffeomorphism onto its image that maps $X_x$ onto its \em tangent cone\em, that is, the \em union \em of the tangent spaces of its irreducible components; henceforth, by abuse of notation, tangent cone will also mean the \em collection \em of those tangent spaces.

In case $X\cap U=\{u_{\ell_1}=0\}\cup\cdots\cup\{u_{\ell_r}=0\}$, $1\leq \ell_1,\ldots,\ell_r \leq m$, we have a \em Nash normal crossings \em  \cite{fgr}. After the obvious linear change of coordinates we can assume that $X\cap U=\{u_{1}=0\}\cup\cdots\cup\{u_{r}=0\}$. That is, in the context of Nash normal crossings, the number of hyperplanes determines the type up to linear isomorphism. For monomial singularities the characterization of the type is far more involved. Our first aim will be to understand when a family of linear varieties of $\mathbb{R}^m$ is linearly isomorphic to a family of coordinate linear varieties.

To that end, let $\mathcal{L}=\{L_1,\ldots,L_s\}$ be a family of linear varieties of $\R^m$. Henceforth, every 
time we consider a family of linear varieties we assume there are no immersions. 

For each subset $I\subset \{1,\ldots,s\}$ and each $1\leq p\leq s$ we denote 
$$
L_I=\bigcap_{j\in I} L_{j} \quad \text{and}\quad L^{(p)}=\sum_{\# I=p}L_I.
$$
We set $L^{(s+1)}=\{0\}$. For each $I\subset \{1,\ldots,s\}$ with $\# I=p$ we also define  $$V_I=L^{(p+1)}\cap L_I$$
and we denote with $W_I$  any \em supplement \em of $V_I$ in $L_I$. We will use this notation consistently in all what follows. The following equations hold:

\begin{stepsp}\label{eq:ecu1}$L^{(p)}=L^{(p+1)}+\sum_{\#I=p}W_I=\sum_{k=p}^s \sum_{\#J=k}W_J.$
\end{stepsp}

\begin{stepsp}\label{eq:ecu2}$\sum_{\#J=p, \, J\neq I}  L_J\cap L_I \subset V_I$ \, and \, $L_I\supset \sum_{J\supset I} W_J$. 
\end{stepsp}

\begin{stepsp}\label{eq:ecu3}$\dim(W_I)=\dim(L_I)-\dim(V_I)\leq\dim(L_I)-\dim\Big(\sum_{\substack{\# J=p \\ J\neq I}} \, L_J\cap L_I\Big).$ 
\end{stepsp}

\begin{stepsp}\label{eq:ecu4}$\dim(L^{(p)})\leq\dim(L^{(p+1)})+\sum_{\# I=p}\Big(\dim(L_I)-\dim\Big(\sum_{\substack{\# J=p\\ J\neq I}}\, L_J\cap L_I\Big)\Big)$.
\end{stepsp}

Indeed, \ref{eq:ecu2} and \ref{eq:ecu3} are obvious. To prove \ref{eq:ecu1} we pick $u\in L^{(p)}$ and write $u=\sum_{\# I=p}u_I$ where $u_I\in L_I$. Next, write $u_I=v_I+w_I$ where $v_I\in V_I$ and $w_I\in W_I$. Since $\sum_{\# I=p}V_I\subset L^{(p+1)}$ we have
$$u=\sum_{\# I=p}v_I+\sum_{\# I=p}w_I\in L^{(p+1)}+\sum_{\# I=p}W_I$$
and so $L^{(p)}=L^{(p+1)}+\sum_{\# I=p}W_I$, as required. Finally \ref{eq:ecu4} follows from \ref{eq:ecu2} and \ref{eq:ecu3}.

The last inequality \ref{eq:ecu4} gives way to the following notion.
\begin{defn}We say that $\mathcal{L}$ is an \em extremal family \em if the inequality \ref{eq:ecu4} is an equality for all $p=1,\ldots,s$.
\end{defn}
Clearly, this notion does not depend on the ordering the varieties are listed in the family $\mathcal{L}$.
\begin{lem}\label{iniprop}Suppose ${\mathcal L}$ extremal. Then for $p=1,\ldots,s$ and $I\subset\! \{1,\ldots,s\}$ with $\#I\!=p$ we have:

\em (1) \em $L^{(p)}=L^{(p+1)}\oplus\bigoplus_{\# I=p}W_I=\bigoplus_{k=p}^s \, \bigoplus_{\#J=k}W_J$.

\em (2) \em $V_I=\sum_{\substack{\# J=p\\ J\neq I}}\,L_J\cap L_I$.

\em (3) \em $L_I=\bigoplus_{J\supset I}W_J$.

In particular, $L^{(1)}=\bigoplus_J W_J$ and if we choose a basis $\bs_J$ for each $W_J$, then any extension of the union of the $\bs_J$'s to a basis $\bs$ of $\R^m$ satisfies that $\bs\cap L_I$ is a basis of $L_I$ for every $I$.
\end{lem}
\begin{proof}(1) By \ref{eq:ecu1} we have $L^{(p)}=L^{(p+1)}+\sum_{\# I=p}W_I$ and therefore it is enough to check that $\dim(L^{(p)})=\dim(L^{(p+1)})+\sum_{\# I=p}\dim(W_I)$. Now, by \ref{eq:ecu2} and \ref{eq:ecu3} and since the family is extremal we have
\begin{multline*}
\dim(L^{(p)})\leq\dim(L^{(p+1)})+\sum_{\# I=p}\dim(W_I)\leq \\ \dim(L^{(p+1)})+\sum_{\# I=p}\Big(\dim(L_I)
-\dim\Big(\sum_{\# J=p,\, J\neq I}\!\!\!L_J\cap L_I\Big)\Big)=\dim(L^{(p)}).
\end{multline*}
so that all the inequalities are equalities, as required.

(2) Since $\sum_{\# J=p,J\neq I}L_J\cap L_I\subset V_I$ it is enough to check dimensions coincide. By \ref{eq:ecu1} 
and the last equation in the proof of (1) we deduce $\dim(W_I)=\dim(L_I)-\dim\Big(\sum_{\substack{\# J=p \\ J\neq I}}\,L_J\cap L_I\Big)$
and so
$$
\dim(V_I)=\dim(L_I)-\dim(W_I)=\dim\Big(\sum_{\# J=p,\, J\neq I}\!\!\!L_J\cap L_I\Big).
$$
\indent(3) In view of (1) it suffices to show that $L_I=\sum_{J\supset I}W_J$. We proceed by descending induction on $\#I$, being the first step $I=\{1,\ldots,s\}$ obvious. 
Assume the result true for $\#I> p$  and let us check it  for $\#I=p$. By definition $L_I=V_I\oplus W_I$ and by (2) we have 
$$
V_I=\sum_{\# J=p,\, J\neq I}\!\!\!\! L_J\cap L_I
$$
For each $J\neq I$ with $\# J=p$ we have $\#(J\cup I)>p$ and by induction hypothesis 
$L_J\cap L_I=L_{J\cup I}=\sum_{K\supset J\cup I}W_K$ and so
$$
L_I=W_I+V_I=W_I+\!\!\! \sum_{\# J=p, \, J\neq I}\!\!\!\! L_J\cap L_I=W_I+ \!\!\! \sum_{\# J=p, \, J\neq I} \, \sum_{\, K\supset J\cup I}W_K=\sum_{J'\supset I}W_{J'},
$$ 
as claimed.

For the final assertion of the statement, note that $L^{(1)}=\bigoplus_JW_J$ follows straightforward from (1), and hence we can indeed obtain 
$\bs$ as explained. Furthermore, by (3), $\bs\cap L_I$ is the union of all $\bs_J$ with $J\supset I$ and a basis of $L_I$.
\end{proof}

We will prove that extremal families are linear isomorphic to families of coordinate linear varieties. To that aim we introduce the following definition.

\begin{defn}Let $\mathcal{L}=\{L_1,\ldots,L_s\}$ be a family of linear varieties of $\R^m$. We say that a basis $\bs$ of $\R^m$ is \em adapted to $\mathcal{L}$ \em  if for all $i=1,\ldots,s$ the intersection $\bs\cap L_i$ is a basis of $L_i$.
\end{defn}
For example, for any family of coordinate linear varieties of $\R^m$, the standard basis $\bes$ of $\R^m$ is an adapted basis of the family. Moreover, note that if a family $\mathcal{L}=\{L_1,\ldots,L_s\}$ admits an adapted basis $\bs$ then any bijection $\bs \leftrightarrow \bes$ induces a linear isomorphism $f:\R^m\rightarrow \R^m$ such that $f(L_i)$ is a coordinate linear variety of $\R^m$. Thus, the following is the result we were interested in:
\begin{prop}\label{prop:extremal}
Let $\mathcal{L}$ be a family of linear varieties of $\mathbb{R}^m$. Then ${\mathcal L}$ is a extremal family of $\R^m$ if and only if it there is a basis of $\R^m$ adapted to $\mathcal{L}$. 
\end{prop}
\begin{remark}\label{rmk:adapt}If $\bs$ is an adapted basis of $\mathcal{L}=\{L_1,\ldots,L_s\}$ then (i) $\bs \cap L_I$ is a basis of $L_I$ for all $I\subset\{1,\ldots,s\}$, and (ii) $(\sum_{i=1}^\ell L_{J_i})\cap L_I=\sum_{i=1}^\ell(L_{J_i}\cap L_I)$ for all $J_1,\ldots,J_\ell,I\subset\{1,\ldots,s\}$. Indeed, it is enough to notice that if ${\mathcal G}_1,{\mathcal G}_2\subset\bs $ then $L[{\mathcal G}_1]\cap L[{\mathcal G}_2]=L[{\mathcal G}_1\cap{\mathcal G}_2]$ and $L[{\mathcal G}_1]+L[{\mathcal G}_2]=L[{\mathcal G}_1\cup{\mathcal G}_2]$. 
\end{remark}
\begin{proof}[Proof of Proposition \ref{prop:extremal}] We proved one direction in Lemma \ref{iniprop}. Now, assume that $\mathcal{L}$ admits an adapted basis $\bs $ and denote $\bs _i=\bs \cap L_i $ the basis of $L_i$ for $i=1,\ldots,s$. Let us check that ${\mathcal L}$ is extremal. Recall that $L^{(p)}=L^{(p+1)}+\sum_{\# I=p}W_I$ (see \ref{eq:ecu1}) and let us check that the previous sum is direct. It is enough to see that vector $0$ only admits the trivial representation as a sum of vectors of the linear varieties $L^{(p+1)}$ and $W_I$ where $\#I=p$. Indeed, write
$$
0=u_{p+1}+\sum_{\# I=p}w_I \quad \text{ with } u_{p+1}\in L^{(p+1)},\ w_I\in W_I.
$$
Then for any $I$ with $\#I=p$  we have
$$-w_I=u_{p+1}+\!\!\! \sum_{\# J=p, \, J\neq I} \!\!\! w_J \, \in \, \Big(L^{(p+1)}+\!\!\! \sum_{\# J=p, \, J\neq I} \!\!\! W_J\Big)\cap W_I.
$$
By Remark \ref{rmk:adapt},
$$
\Big(L^{(p+1)}+ \!\!\!\! \sum_{\# J=p, \, J\neq I} \!\!\! W_J\Big)\cap W_I\subset \Big(L^{(p+1)}+ \!\!\! \sum_{\# J=p, \, J\neq I} \!\!\!\! L_J\Big)\cap L_I=
L^{(p+1)}\cap L_I+\!\!\! \sum_{\# J=p, \, J\neq I}\!\!\!\! L_J\cap L_I\subset L^{(p+1)}\cap L_I 
$$
so that $-w_I\in (L^{(p+1)}\cap L_I)\cap W_I=V_I\cap W_I=\{0\}$ and so also $u_{p+1}=0$, as required. In particular, we deduce that
$$
\dim(L^{(p)})=\dim(L^{(p+1)})+\sum_{\# I=p}\dim(W_I).
$$
Thus, since $\dim(W_I)=\dim(L_I)-\dim(V_I)$, to prove that $\mathcal{L}$ is extremal it is enough to show that
$$\dim(V_I)=\dim\Big(\sum_{\# J=p , \, J\neq I}\!\!\!\! L_J\cap L_I\Big).$$
But since $\mathcal{L}$ admits an adapted basis we have (see Remark \ref{rmk:adapt})
$$
V_I=L^{(p+1)}\cap L_I= \Big(\sum_{\# K=p+1}\!\!\!\! L_K\Big)\cap L_I= \!\!\! \sum_{\# K=p+1}\!\!\!\! L_K\cap L_I \, \subset \!\!\! \sum_{\# J=p, \, J\neq I} \!\!\!\! L_J\cap L_I\subset L^{(p+1)}\cap L_I=V_I,
$$
as required.
\end{proof}

\begin{remark}\label{sperner} Let $\mathcal{ L}$ be an extremal family of linear varieties (without immersions). Once we know it is, up to linear isomorphism, a family of coordinate linear varieties, we can bound the number of varieties in $\mathcal L$. Indeed, by associating to every coordinate variety $L\subset\R^m$ the set $\{x_{i_1},\dots,x_{i_r}\}$ of the variables appearing in the equations of $L$ we define a bijection from the set of all coordinate linear varieties in $\R^m$ onto the set of all subsets of $\{x_1,\dots,x_m\}$. Clearly, this bijection reverses inclusions, hence transforms $\mathcal L$ in an \em Sperner family \em of a finite set of $m$ elements. Now, it is a beautiful result, the \em Sperner Theorem \em \cite{l}, that such a family has at most
$\binom{m}{[m/2]}$ elements. Thus
$$
\#(\mathcal{L})\le\binom{m}{[m/2]}.
$$
This is behind the value $q$ in Theorem  \ref{approxgc} (see the final step in its proof \ref{pfapproxgc}).
\end{remark}

Now, we need to determine whether two extremal families are linearly isomorphic. We introduce the following general definition. 
\begin{defn}Let $\mathcal{L}=\{L_1,\ldots,L_s\}$ be a family of linear varieties of $\R^m$. As before, we denote for each subset $I\subset \{1,\ldots,s\}$ the intersection $L_I=\bigcap_{j\in I} L_{j}$. Next, to each family of different nonempty subsets $I_1,\ldots,I_r\subset\{1,\ldots,s\},\ r\geq1$, we associate the number
$$
\dim(L_{I_1}+\cdots+L_{I_r}).
$$
The collection of all the previous dimensions will be called the \em load \em of $\mathcal{L}$.
\end{defn}

Note that this notion of load does depend on the ordering the varieties are listed in the family $\mathcal{L}$. Also note that the combinatorial information to determine if a family ${\mathcal L}$ is extremal is contained in the load of ${\mathcal L}$. 

\begin{prop}\label{prop:loadiso}
Let ${\mathcal L}=\{L_1,\ldots,L_s\}$ and ${\mathcal L}'=\{L_1',\ldots,L_s'\}$ be two extremal families of $\R^m$. Then, there is a linear isomorphism $f$ of $\R^m$ such that $f(L_i)=L_i'$ for all $i$ if and only if the loads of ${\mathcal L}$ and ${\mathcal L}'$ coincide. \em If this is the case, we say that the families ${\mathcal L}$ and ${\mathcal L}'$ are \em equivalent.
\end{prop}
\begin{proof}
Assume that the loads of ${\mathcal L}$ and ${\mathcal L}'$ coincide. By Lemma \ref{iniprop} we have that $L^{(1)}=\bigoplus_J W_J$ and that to construct a basis $\bs$ of $\R^m$ adapted to ${\mathcal L}$ it is enough to choose a basis $\bs_I$ of $W_I$ for each $I\subset\{1,\ldots,s\}$ and then to extend the union of all $\bs_I$'s to obtain $\bs$. Similarly, we construct a basis $\bs'$ of $\R^m$ adapted to ${\mathcal L}'$ that extends a union of the bases $\bs'_I$'s of the $W'_I$'s.  Since the loads of ${\mathcal L}$ and ${\mathcal L}'$ coincide, $\bs _I$ and $\bs _I'$ have the same number of elements for each $I\subset\{1,\ldots,s\}$. Thus, bijections $\bs_I \leftrightarrow \bs'_I$ induce another $\bs \leftrightarrow \bs'$. Hence, there is a linear isomorphism $f$ of $\R^m$ that maps $W_I$ onto $W_I'$ for all $I$.  Since again by Lemma \ref{iniprop} we know that $L_i=\bigoplus_{i\in I} W_I$ and $L'_i=\bigoplus_{i\in I} W'_I$, we are done. The converse implication is clear since isomorphisms preserve dimensions.
\end{proof}

\begin{remarks}\label{varios}(1) We have a necessary condition for a union of non-singular germs to be a monomial singularity: its tangent cone must be extremal (for a full characterization, just involving arithmetic conditions, see Corollary \ref{cor:sufficient}). Note that given a collection of linear varieties, the mere list of their dimensions and the dimensions of their intersections is not enough to determine if they are a monomial singularity: for instance, three lines through the origin in $\R^3$ are a monomial singularity if and only if they generate $\R^3$, hence we cannot drop the dimension of the sum of the lines. 

(2) Since the equivalence of extremal families of linear varieties is determined by their loads, it is an arithmetic relation as mentioned in the Introduction.

(3) Let $X_x$ and $X'_x$ be two monomial singularities. Then, $X_x$ and $X'_x$ are Nash isomorphic if and only if their types $\varLambda$ and $\varLambda'$ are equivalent if and only if their tangent cones have up to reordering the same load. Thus, in the end, the type of a monomial singularity is characterized arithmetically by the load of the tangent cone.
\end{remarks}

\section{Nash monomial singularity germs}
Here we will prove the semialgebraicity result stated in Proposition \ref{satype}. But previously we must analyse the ideals of monomial singularities.
\subsection{Square-free monomial ideals.}\label{sqfree} 
In this section $\sD$ will either denote one of the domains: (i) $\sn(\R^m)$ of global Nash functions on $\R^m$, (ii) $\sn_m$ of Nash function germs at the origin in $\R^m$ and (iii) $\so_m$ of analytic function germs at the origin in $\R^m$.  A \em square-free monomial ideal \em of $\sD$ is an ideal generated by monomials $x^\sigma=x_1^{\sigma_1}\cdots x_m^{\sigma_m}$ with exponents $\sigma_i=0$ or $1$. Our aim here is to show that: \em the ideals of unions of coordinate linear varieties are exactly the square-free monomial ideals. \em To that end, it will be useful the following lemma.

\begin{lemsub}\label{key}
Let $I\subset \sD$ be a proper ideal that admits a system of generators $f_1,\ldots,f_\ell$ that do not depend on $x_1$. Then, $x_1$ is a non zero divisor $\mod I$.
\end{lemsub}
\begin{proof}
We have to show that if $x_1g\in I$ for some $g\in \sD$, then $g\in I$. But, if 
$$
x_1g=h_1f_1+\cdots+h_\ell f_\ell \text{ for some } h_i\in \sD,
$$
we can write $h_i=x_1q_i+r_i$, where $r_i=h_i(0,x_2,\ldots,x_m)$ does not depend on $x_1$, and 
$$
q_i=\frac{h_i-r_i}{x_1}\in \sD.
$$
Therefore,
$$
x_1g=x_1(q_1f_1+\ldots+q_\ell f_\ell)+(r_1f_1+\ldots+r_\ell f_\ell).
$$
If we make $x_1=0$, and since $r_i,f_i$ do not depend on $x_1$, we deduce that
$$
r_1f_1+\ldots+r_\ell f_\ell=0$$
and hence $x_1g=x_1(q_1f_1+\ldots+q_\ell f_\ell)$. As $\sD$ is a domain we conclude that $g=q_1f_1+\ldots+q_\ell f_\ell\in I$, as desired. 
\end{proof}

\begin{defnsub}\label{def:squarefree}Let $X=L_1\cup \cdots \cup L_s$ be a union of coordinate linear varieties of $\R^m$. We say that a \em $x_i$ is a variable of $L_j$ \em if $x_i=0$ is one of the equations of $L_j$. Now, we first consider the collection of all monomials $x_{j_1}\cdots x_{j_s}$ such that $x_{j_i}$ is a variable of $L_i$ for each $i\in \{1,\ldots,s\}$; then, if a variable of $x_{j_1}\cdots x_{j_s}$ appears several times (because it comes from several $L_i$'s) we just take it once. The resulting monomials are the \em associated square-free monomials of $X$\em.
\end{defnsub}

\begin{propsub}\label{sqfreestate}Let $X=L_1\cup\cdots\cup L_s$ be a union of coordinate linear varieties of $\R^m$. Then its associated square-free monomials generate: \em (i) \em  the ideal $I(X)$ of global Nash functions vanishing in $X$, \em (ii) \em the ideal $I(X_0)$ of Nash function germs vanishing on the germ at the origin $X_0$ and \em (iii) \em the ideal $I_{\so}(X_0)$ of analytic function germs vanishing on $X_0$.
\end{propsub}

\begin{proof}All cases are the same, so we write down (i). We first show that $X$ is the zero set of the system defined by its associated square-free monomials. Indeed, if $L_i$ has equations $x_{i_1}=\cdots=x_{i_r}=0$ we can use instead 
$x_{i_1}^2+\cdots+x_{i_r}^2=0$, so that $X$ is given by
$$
0={\prod}_i(x_{i_1}^2+\cdots+x_{i_r}^2)={\sum}_\ell({\prod}_ix_{i_\ell})^2,
$$
and so the equations ${\prod}_ix_{i_\ell}=0$ define $X$. Of course, we can eliminate repetitions of variables in $\prod_i x_{i_\ell}$ to get the equations given by the associated square-free monomials. 

By the real Nullstellensatz it remains to show that every square-free monomial ideal $I$ is real. To that end, it is enough to prove that:

\em 
A square-free monomial ideal $I$ is an intersection of prime ideals generated by subsets of variables appearing in the given generators of $I$. 
\em

We argue by induction on the total number of times the variables occur in the given generators following \cite{hd}. For instance, in $(x_1x_2,x_1x_3)$ there are $4$ occurrences. In the induction step we will use Lemma \ref{key}. For one occurrence there is nothing to prove, and the same happens for the more general case when each generator is one single variable. Suppose now some variable appears in all monomials, say $I=(x_1f_1,\dots,x_1f_p)$ where the $f_k$'s are square-free monomials without $x_1$. For the induction we only need to show $I=(x_1)\cap (f_1,\dots,f_p)$. Here, the inclusion left to right is clear. For the other suppose $x_1h\in 
(f_1,\dots,f_p)$. By Lemma \ref{key}, the variable $x_1$ is not a zero-divisor $\mod(f_1,\dots,f_p)$, hence this ideal contains $h$ and consequently $x_1h\in I$. Thus we can assume that some generator has at least two variables and that no variable appears in all of them; say $I=(x_1f_1,\ldots,x_1f_p,g_1,\ldots,g_q)$, $f_1\ne1$ and no $g_j$ contains $x_1$. We claim that
$$
I=(x_1,g_1,\ldots,g_q)\cap(f_1,f_2,\ldots,f_p,g_1,\ldots,g_q),
$$
which gives way to induction again. The inclusion to prove is right to left, which reduces to: \em if $x_1h\in (f_1,f_2,\ldots,f_p,g_1,\ldots,g_q)$ then $x_1h\in I$\em. But, by Lemma \ref{key}, $x_1$ is not a zero-divisor $\mod (f_1,f_2,\ldots,f_p,g_1,\ldots,g_q)$, and consequently $h\in (f_1,f_2,\ldots,f_p,g_1,\ldots,g_q)$ so that $x_1h\in I$, as claimed.
\end{proof}
From this result we directly obtain:
\begin{corsub}\label{cor:type}
A Nash germ $X_x$ is a monomial singularity if and only if the ideal $I(X_x)\subset \sn_{M,x}$ is generated by square-free monomials on some local coordinates at $x$.
\end{corsub}

After this discussion of square-free monomials ideals we can turn to:

\subsection{Semialgebraicity of monomial singularities loci.}\label{subsec:loci}
We know this to be true for normal crossings \cite[1.5]{fgr}, and we are to generalize the argument there. But we give full details because of its technical nature.

\begin{proof}[Proof of Proposition \ref{satype}]\label{proofsatype} We will prove the semialgebraicity of $T^{(\varLambda)}$ for any fixed type $\varLambda$. Consider the Nash ideal $I=I(X)$ of $X$, which is finitely generated by some $f_1,\dots,f_p\in\sn(M)$. By Corollary \ref{cor:type} there are square-free monomials $m^{(\varLambda)}_1(\x),\ldots,m^{(\varLambda)}_r(\x)\in\Z[\x_1,\ldots,\x_m]$ depending only on $\varLambda$ such that $x\in T^{(\varLambda)}$ if and only if 

\vspace{1mm}\setcounter{substep}{0}
\begin{substeps}{proofsatype}\label{ast2}
\em There is a regular system of parameters $u=(u_1,\dots,u_m)$ of the local regular ring $\sn_{M,x}$ such that $X_x=\{f_1=0,\ldots,f_p=0\}_x=\{m^{(\varLambda)}_1(u)=\cdots=m^{(\varLambda)}_r(u)=0\}_x$.
\end{substeps}

\vspace{1mm}
We must show that this condition is semialgebraic. Before proceeding, we apply the Artin-Mazur Theorem \cite[8.4.4]{bcr} to assume that $M$ is an open subset of a nonsingular algebraic set $V\subset\R^n$ and $f_1,\dots,f_p$ are the restrictions to $M$ of some polynomial functions that we denote by the same letters. Let $J$ be the ideal of $V$ in the polynomial ring $\R[\x]=\R[\x_1,\ldots,\x_n]$, and let $b_1,\dots,b_q$ be generators of $J$. By \cite[8.7.15]{bcr} the stalk $\sn_{M,x}$ at a point $x\in M$ is the henselization of the localization of $\R[\x]/J$ at the ideal $(\x-x)=(\x_1-x_1,\dots,\x_n-x_n)$. The henselization of the local ring $\R[\x]_{(\x-x)}$ is $\R[[\x-x]]_\text{alg}$ and so $\sn_{M,x}=\R[[\x-x]]_\text{alg}/J_x$ where $J_x=J\R[[\x-x]]_\text{alg}$. Therefore, the parameters $u_i$ are the classes modulo $J_x$ of some $h_i\in\R[[\x-x]]_\text{alg}$; let $B_{kx},H_{ix}$ stand for the derivatives at $x$ of the $b_k,h_i$.
 
Suppose that condition (\ref{proofsatype}.\ref{ast2}) holds for a point $x\in M$. We deduce that all $f_j$'s belong to the ideal of $\sn_{M,x}$ generated by $m^{(\varLambda)}_1(u),\ldots,m^{(\varLambda)}_r(u)$; hence 
\begin{enumerate}
\item[(1)] There are Nash function germs $g_{ej},a_{jk}\in\R[[\x-x]]_\text{alg}$ such that 
\begin{equation}\label{eq:star}
f_j={\sum}_em^{(\varLambda)}_e(h_1,\ldots,h_m)g_{ej}+{\sum}_k a_{jk}b_k.
\end{equation}
\end{enumerate}
On the other hand, that the $u_i$'s form a regular system of parameters of $\sn_{M,x}$ just means that
\begin{enumerate}
\item[(2)] The $h_i$'s vanish at $x$, and the linear forms $H_{ix}$ are linearly independent over $\R$ modulo the linear forms $B_{kx}$.
\end{enumerate}

Let us now see how these new two conditions are semialgebraic. We look at equation \eqref{eq:star} as a system of polynomial equations in the unknowns ${\tt h}_i,{\tt g}_{ej},{\tt a}_{jk}$. Then we recall M. Artin's approximation theorem with bounds \cite[6.1]{ar}: 

\vspace{1mm}
\begin{substeps}{proofsatype}\em
For any integer $\alpha$ there exists another integer $\beta$ which only depends on $n,\alpha,$ the degrees of the $f_j$'s,  the degrees of the $m^{(\varLambda)}_e$'s, the degrees of the $b_k$'s and the number of variables ${\tt h}_i,{\tt g}_{ej},{\tt a}_{ij}$, such that the polynomial equations
$$
f_j={\sum}_em^{(\varLambda)}_e({\tt h}_1,\ldots,{\tt h}_m){\tt g}_{ej}+{\sum}_k{\tt a}_{jk}b_k
$$
have an exact solution in the local ring $\R[[\x-x]]_{\rm alg}$ if they have an approximate solution modulo $(\x-x)^{\beta}$; furthermore that exact solution coincides with the approximate solution till order $\alpha$. 
\end{substeps}

\vspace{1mm}
Now, fix $\alpha=2$, so that the exact solution coincides with the approximate one till order $2$, and define $S$ as the set of points $x\in M$ such that:
\begin{itemize}
\item[(1*)] There are polynomials $\mathrm{h}_i,\mathrm{g}_{ej},\mathrm{a}_{jk}\in\R[\x]$ of degree $\leq\beta$ such that
\begin{equation}\label{eq:ast}
f_j\equiv{\sum}_em^{(\varLambda)}_e(\mathrm{h}_1,\ldots,\mathrm{h}_m)\mathrm{g}_{ej}+{\sum}_k\mathrm{a}_{jk}b_k\mod(\x-x)^{\beta}.
\end{equation}
\item[(2*)] The polynomials $\mathrm{h}_i$ vanish at $x$ and their derivatives $\mathrm{H}_{i,x}$ at $x$ are linearly independent modulo the linear forms $B_{kx}$. 
\end{itemize}

Notice that if the approximate solution $\mathrm{h}_i$ verifies (2*) then the exact one $h_i$ verifies (2). Thus, if the equation \eqref{eq:ast} has an \em approximate \em solution $\mathrm{h}_i,\mathrm{g}_{ej},\mathrm{a}_{jk}\in\R[\x]$ of degree $\leq\beta$ modulo $(\x-x)^{\beta}$ satisfying (2*), then the equation \eqref{eq:star} has an \em exact \em solution $h_i,g_{ej},a_{jk}\in\R[[\x-x]]_\text{alg}$ satisfying (2). Since the converse implication is trivial, both assertions are equivalent. Now, the existence of approximate solutions of fixed order $\beta$ (described by conditions (1*) and (2*) above) is a first order sentence, and we conclude that the set $S$ of points $x\in M$ for which conditions (1) and (2) hold true (or equivalently conditions (1*) and (2*) hold) is a semialgebraic set. 

Next we analyze the exact meaning of (1) and (2); let $x\in S$. From (1) we get that $$\{m^{(\varLambda)}_1(u)=\cdots=m^{(\varLambda)}_r(u)=0\}_x\cap M_x\subset \{f_1=0,\ldots,f_p=0\}_x\cap M_x$$ and therefore
\begin{multline}\label{eq:bullet}
X_x=\{f_1=0,\ldots,f_p=0\}_x\cap M_x= \\ =\Big( \{m^{(\varLambda)}_1(u)=\cdots=m^{(\varLambda)}_r(u)=0\}_x\cap M_x\Big) \cup(Y_{1,x}\cup\cdots\cup Y_{s,x}),
\end{multline}
where the $Y_{\ell,x}$'s are the irreducible Nash components of $X_x$ on which some $m^{(\varLambda)}_i(u)$ does not vanish identically. Thus we must get rid of those $Y_{\ell,x}$'s. To that end we use the topology of the germ $X_x$.

Let us denote $X'_x=\bigcup_{\lambda\in \varLambda}\{u_\lambda=0\}_x \cap M_x$. Recall that by definition of the monomials $m^{(\varLambda)}_1(\x),\ldots,m^{(\varLambda)}_r(\x)$ we also have 
$$X'_x=\{m^{(\varLambda)}_1(u)=\cdots=m^{(\varLambda)}_r(u)=0\}_x\cap M_x.$$ For every integer $d$ denote $c_d$ the number of connected components of dimension $d$ of the
smooth locus of $X'_x$; this number only depends on $\varLambda$ and $d$. Now consider the semialgebraic set
$$
{\tt Smooth}_d(X)=\{x\in X: \text{the germ $X_x$ is smooth of dimension $d$}\}.
$$
We know that the set
$$
C_d=\{x\in X:\text{the germ ${\tt Smooth}_d(X)_x$ has $c_d$ connected components}\}
$$
is semialgebraic \cite[4.2]{fgr}, and so is the intersection $C=\bigcap_dC_d$. To conclude, we see that $S\cap C=T^{(\varLambda)}$.

Clearly $S\cap C$ contains $T^{(\varLambda)}$. For the other inclusion, pick $x\in S\cap C$. We must see that the $Y_{\ell,x}$'s in equation \eqref{eq:bullet} are redundant. Otherwise, let $d$ be the biggest dimension of the non-redundant ones. We can write
$$
{{\tt Smooth}_d(X)}_x=
 \Big({\tt Smooth}_d (X'_x)
 \setminus {\bigcup}_\ell Y_{\ell,x}\Big)\cup
 \Big({\tt Smooth}_d\Big({\bigcup}_\ell Y_{\ell,x}\Big)\setminus
 X'_x\Big),
$$
where only non-redundant $Y_{\ell,x}$'s are considered. As $x\in C_d$, in the left hand side we have exactly $c_d$ connected components. In the right hand side, we have at least $c_d$ coming from the first bracket. This is so because we have $c_d$ connected components from ${\tt Smooth}_d(X'_x)$ by definition of $c_d$, and none of them can be lost inside $\bigcup_\ell Y_{\ell,x}$. Indeed, if one of these connected components, say $E$, was contained in the union ${\bigcup}_\ell Y_{\ell,x}$, it would be contained in some of the $Y_{\ell,x}$ of dimension $d$ (dimension cannot be bigger by construction). The connected component $E$ is an open subset of $\{u_{\lambda}=0\}_x\cap M_x$ for some $\lambda\in \varLambda$ and therefore $\{u_{\lambda}=0\}_x\cap M_x= Y_{\ell,x}$ because both are irreducible and have the same dimension. In particular $Y_{\ell,x}$ is redundant, a contradiction. Now, looking at the second bracket, if $Y_{\ell,x}$ is not redundant of dimension $d$, the germ
$$
{\tt Smooth}_d(Y_{\ell,x})\setminus X'_x
$$
is not empty, and adds some connected component, which is impossible. This contradiction completes the argument.
\end{proof}

\subsection{Nash functions on Nash monomial singularity germs.} 
In Definition \ref{def:nashfunc} we introduced Nash and $\tt c$-Nash functions. We now prove that both concepts coincide for Nash monomial singularity germs. This will allow in the next section to show the corresponding result for Nash sets with monomial singularities (see Theorem \ref{cnashgerms}).

\begin{propsub}\label{trick}
Let $X=L_1\cup\cdots\cup L_s$ be a union of coordinate linear varieties in $\R^m$, and let $h:X\to\R$ be a semialgebraic ${}^{\tt c}\sS^\nu$ (resp. $^{\tt c} \sn$) function. For every non-empty set of indices $I$ we consider the intersection $L_I=\bigcap_{i\in I}L_i$ and the orthogonal projection $\pi_I:\R^m\to L_I$. Then the function
$$
H={\sum}_I(-1)^{\#I+1}h\circ\pi_I.
$$
is a well defined $\sS^\nu$ (resp. Nash) extension of $h$ to $\R^m$. 

In particular, $\sS^\nu(X)\equiv{}^{\tt c}\sS^\nu(X)$ and $\sn(X)\equiv{}^{\tt c}\sn(X)$.
\end{propsub}
\begin{proof} 
We write the proof for the differentiable case, as the Nash one is a copy.
First note that every $L_I$ is contained in some $L_i$, hence $h\circ\pi_I=(h|L_i)\circ\pi_I$ is a composition of semialgebraic $\sS^\nu$ functions, hence a semialgebraic $\sS^\nu$ function. Consequently, $H$ is a sum of semialgebraic $\sS^\nu$ functions, and so a semialgebraic $\sS^\nu$ function. Thus the thing to check is that $H$ extends $h$, which we prove by induction on $s$. 

For $s=1$, we just have $H=h\circ\pi_1$, where $\pi_1$ is the orthogonal projection onto $X=L_1$. Thus, $\pi_1|_X=\Id_X$ and $H|_X=h$. Now suppose $s>1$ and the result true for $s-1$ coordinate linear varieties. Denote $Y=L_1\cup\cdots\cup L_{s-1}$, and
$$
G={\sum}_{s\notin J}(-1)^{\#(J)+1}h\circ\pi_J,
$$
that is, $J$ runs among the non-empty sets of indices that do not contain $s$. The induction hypothesis says that $G|_Y=h|_Y$, that is, $h-G$ vanishes on $Y$. Let us deduce from this that $H|_X=h$.

Denote $\pi_s:\R^m\to L_s$ the orthogonal projection onto $L_s$. Since the $L_i$'s are coordinate linear varieties, $\pi_s(L_i)=L_i\cap L_s\subset L_i$, which implies that $\pi_s(Y)\subset Y$ and that $(h-G)\circ\pi_s$ vanishes on $Y$ as $h-G$ does. On the other hand, $(h-G)\circ\pi_s$ and $h-G$ coincide on $L_s$, because $\pi_s$ is the identity on $L_s$. Thus $(h-G)-(h-G)\circ\pi_s$ vanishes on $X$, that is, 
$$
h=G+(h-G)\circ\pi_s\quad\text{on}\quad X.
$$
But:
\begin{align*}
(h-G)\circ\pi_s&=h\circ\pi_s-{\sum}_{s\notin J}(-1)^{\#(J)+1}h\circ\pi_J\circ\pi_s\\
 &=h\circ\pi_s-{\sum}_{s\notin J}(-1)^{\#(J)+1}h\circ\pi_{J\cup\{s\}}
 =h\circ\pi_s+{\sum}_{s\notin J}(-1)^{\#(J)+2}h\circ\pi_{J\cup\{s\}}\\
 &={\sum}_{s\in K}(-1)^{\#(K)+1}h\circ\pi_K , 
\end{align*}
that is, $K$ runs among the non-empty subsets of indices that do include $s$. Hence \em on $X$ \em we have that
$$
h={\sum}_{s\notin J}(-1)^{\#(J)+1}h\circ\pi_J+{\sum}_{s\in K}(-1)^{\#(K)+1}h\circ\pi_K={\sum}_{I}(-1)^{\#(I)+1}h\circ\pi_I=H,
$$
as required.
\end{proof}

We remark that Proposition \ref{trick} defines a \em extension linear map \em $h\mapsto H$
$$
\sS^\nu(X)\equiv{}^{\tt c}\sS^\nu(X)\to\sS^\nu(\R^m),\quad\sn(X)\equiv{}^{\tt c}\sn(X)\to\sn(\R^m)
$$
for the local models $X=L_1\cup\cdots\cup L_s$. Note however that the above extensions disregard topologies and therefore we must return to this identification later (see Lemma \ref{lem:localextcont}). In any case, we deduce the following. 

\begin{propsub}\label{locweaknor}
Let $X\subset M$ be a set such that $X_x$ is a Nash monomial singularity germ. Then the inclusions $\sn(X_x)\to{}^{\tt c}\sn(X_x)$ and $\sS(X_x)\to{}^{\tt c}\sS(X_x)$ are bijective.
\end{propsub}
\begin{proof}
We prove the Nash case, the $\sS^\nu$ one is similar. It is enough to prove that the inclusion is surjective. Fix a semialgebraic open neighborhood $U$ of $x\in X$ and $u:U\to \R^m$ a Nash diffeomorphism such that $u(X\cap U)$ is a union of coordinates linear varieties $L_1,\ldots,L_s$. Any $f_x\in {}^{\tt c}\sn(X_x)$ is represented by a semialgebraic map $f:V\cap X\to \R$ where $V$ is an open semialgebraic subset of $U$ such that $u(V)$ is a ball $B=B(0,\veps)$ centered at the origin. Consider the Nash diffeomorphism $\psi:B\to \R^m,\ x\mapsto \frac{x}{\sqrt{\veps^2-\|x\|^2}}$, which satisfies $\psi(L_i\cap B)=L_i$. The map $f\circ u^{-1} \circ \psi^{-1}:L_1\cup \cdots \cup L_s\to \R$ is a $\tt c$-Nash function and by Proposition \ref{trick} it has a Nash extension $H:\R^m\to \R$. Hence $H \circ \psi \circ u:V\to \R$ is a Nash extension of $f$.
\end{proof}

Even though we will not use it until Section \ref{sec:corners} let us write down here the following consequence of Proposition \ref{trick}.

\begin{corsub}\label{cor:extcorners}Let $X=\{x_1\cdots x_s=0 \}\subset \R^m$ and $Q=\{x_1\geq 0,\ldots, x_s\geq 0 \}\subset \R^m$ with $1\leq s \leq m$. Let $f:Q \to \R$ and let $g:X\to \R$ be $\sS^\nu$ functions that coincide on $Q\cap X$. Then, they define a $\sS^\nu$ function on $Q\cup X$, that is, there is an $\sS^\nu$ function $\zeta: \R^m \to \R$ such that $\zeta|_Q=f$ and $\zeta|_X=g$.
\end{corsub}
\begin{proof}Consider $\sS^\nu$ extensions $F,G:\R^m\to \R$ of $f,g$. Then $h=F-G|_X$ vanishes on $Q\cap X$. Let $H:\R^m\to \R$ be the $\sS^\nu$ extension of $h$ given by Proposition \ref{trick}. If $x\in Q$ and $I\subset \{1,\ldots,s\}$ is nonempty then $\pi_I(x)\in Q\cap X$. Thus, by definition $H|_Q=0$ and $\zeta=F-H$ solves the problem. 
\end{proof}

\begin{remarksub}
Observe that for the Nash case, the previous result is trivially true. Given $f$ and $g$ Nash, the unique Nash extension $F$ to $\R^m$ of $f$ coincides, by the Identity Principle, with $g$ on $X$.
\end{remarksub}

\section{Nash sets with monomial singularities}\label{sec:Nashgen}

Recall from the Introduction that a Nash set $X$ of $M$ is a Nash set with monomial singularities if $X_x$ is a monomial singularity for every $x\in X$ (Definition \ref{def:gencrosset}). In this section we prove the finiteness and weak normality (Theorems \ref{thmfiniteness} and \ref{cnashgerms}). First let us state a fundamental fact concerning Nash sets with monomial singularities. 

\begin{lem}\label{lem:irreducohe}
Let $X\subset M$ be a closed semialgebraic set whose germs $X_x$, $x\in X$, are all monomial singularities. Then $X$ is a coherent Nash set. In particular, its irreducible components $X_1,\ldots,X_s$ are pure dimensional and any union of intersections of them is also a Nash set with monomial singularities.
\end{lem}
\begin{proof}
By definition $X$ is locally the zero set of analytic functions. It is coherent in the analytic sense because it is locally a finite union of analytic manifolds \cite[Prop. 13]{cart}. Therefore by Fact \ref{fact:cohpure} it is a global analytic set, hence a Nash set by Fact \ref{fact:globalananash}. Furthermore, by Proposition \ref{fact:cohananash} it follows that $X$ is also coherent in the Nash sense. Let $X_1,\ldots,X_s$ be its irreducible components. By Facts \ref{fact:cohpure} and Proposition \ref{prop:cohirred} they are coherent and pure dimensional. Moreover, $X_{i,x}$ is the union of some irreducible components of $X_x$ for any $x\in X_i\subset X$ (Proposition \ref{prop:cohirred}). Therefore each $X_i$ is a Nash set with monomial singularities. Finally, that any union of intersections of irreducible components is a Nash set with monomial singularities is a local problem, hence reduces to observe that any intersection of coordinate linear varieties is again a coordinate linear variety and any union of coordinate linear varieties is a Nash set with monomial singularities.
\end{proof}
We now prove weak normality making use of the notation and concepts introduced in \ref{subsec:strongcoh}. The corresponding result for $\sS^\nu$ and ${}^{\tt c}\sS^\nu$ functions is also true with the corresponding topologies (Proposition \ref{topirre}), but we must postpone it until finiteness has been proved.

\begin{stepsp}{\em Proof of Theorem \em\ref{cnashgerms}.}\label{pfcnashgerms} Let $\sj$ be the sheaf of $\sn$-ideals given by $\sj_{x}=I(X)\sn_x$. Let $X_1,\ldots,X_s$ be the irreducible components of $X$. For $1 \leq i,j \leq s$ denote  $X_{ij}=X_i\cap X_j$ and let $\sj_i$ and $\sj_{ij}$ be respectively the sheaves of $\sn$-ideals given by $\sj_{i,x}=I(X_i)\sn_x$ and $\sj_{ij,x}=I(X_{ij})\sn_x$. By Lemma \ref{lem:irreducohe} all the Nash sets $X$, $X_i$ and $X_{ij}$ are coherent and therefore we have that $(\sn/\sj)_{x}=\sn_x/ I(X_x)=\sn(X_x)$, $(\sn/\sj_i)_{x}=\sn_x/ I(X_{i,x})=\sn(X_{i,x})$ and $(\sn/\sj_{ij})_{x}=\sn_x/ I(X_{ij,x})=\sn(X_{ij,x})$. In \ref{subsec:strongcoh} we showed that $\sn(X)$ can be naturally identified with the global sections of $\sn/\sj$ and ${}^{\tt c}\sn(X)$ with the global sections of the kernel of the sheaf morphism $\phi:\sn/\sj_1\times \cdots \times \sn/\sj_s\to \prod_{i<j} \sn/\sj_{ij}$ which by coherence at each stalk is $$\phi_x:\sn(X_{1,x})\times \cdots \times \sn(X_{s,x})\to \prod_{i<j} \sn(X_{ij,x}):(f_1,\ldots,f_s)\mapsto (f_i|_{X_{ij}}-f_j|_{X_{ij}})_{i<j}$$ for every $x\in M$. In particular, $\ker (\phi)_x=\ker (\phi_x)=\{(f_1,\ldots,f_s):\ f_i|_{X_{ij}}-f_j|_{X_{ij}}=0, i<j\}$. 

Consider the multiple restriction monomorphism ${\tt i}:\sn/\sj\to \ker(\phi) \subset \sn/\sj_1\times \cdots \times \sn/\sj_s$ given at the level of stalks by 
$$
{\tt i}_x:\sn(X_x)\to \ker (\phi)_x\subset \sn(X_{1,x})\times \cdots \times \sn(X_{s,x}):f\mapsto(f|_{X_{1,x}},\dots,f|_{X_{s,x}}).
$$
We prove that each ${\tt i}_x$ is actually surjective and therefore ${\tt i}:\sn/\sj\to \ker(\phi)$ induces a surjection on global sections $\sn(X)\to {}^{\tt c}\sn(X)$, as required. To prove the surjectivity of ${\tt i}_x$ consider the commutative diagram 
\begin{displaymath}
\xymatrix{
\sn(X_x) \ar[d] \ar[r]^{{\tt i}_x} & \ker(\phi)_x \ar[dl]\\
{}^{\tt c}\sn(X_x) & }
\end{displaymath}
where the vertical arrow is the inclusion and $\ker(\phi)_x\to {}^{\tt c}\sn(X_x)$ is the injective morphism that maps $(f_1,\cdots,f_s)$ to the function germ in ${}^{\tt c}\sn(X_x)$ that equals $f_i$ on $X_{i,x}$. By Proposition \ref{locweaknor} the vertical arrow is an isomorphism and therefore all maps are isomorphisms, and we are done.\fin
\end{stepsp}

We will need the following notion in the proof of the finiteness property (Theorem \ref{thmfiniteness}).
\begin{defn}Let $X\subset M$ be a Nash set with monomial singularities of pure dimension $d$. Then the type $\varLambda$ of $X$ at any point $x\in X$ must consist of sets $\lambda=\{\ell_1,\ldots,\ell_r\}$ of cardinal $r=\dim(M)-d$; the number of those sets is \em the multiplicity of $X$ at $x$ \em and will be denoted $\mult(X,x)$. 
\end{defn}
Note this coincides with the usual multiplicity of the local ring $\so/I_{\so}(X_x)$. For, $I_\so(X_x)$ is clearly the intersection of $\mult(X,x)$ prime ideals of height $\dim(M)-d$ and multiplicity $1$.

We start the proof of Theorem \ref{thmfiniteness} with a geometric preamble.

\begin{prop}\label{fgc}
Let $X\subset M$ be a Nash set with monomial singularities. Then $X$ is the union of finitely many connected Nash manifolds $\varSigma$ on which the type of $X$ is constant and each is contained in an open semialgebraic set $U$ such that 
$$
X\cap U=Y_1\cup\cdots\cup Y_s,
$$
where the $Y_i$'s are closed Nash manifolds in $U$. Furthermore, $\varSigma\subset Y_1\cap\cdots\cap Y_s$ and for 
each $x\in\varSigma$ the germs $Y_{1,x},\dots,Y_{s,x}$ are the irreducible components of $X_x$.
\end{prop}
\begin{proof}
Let $X_1,\dots,X_s$ be the irreducible components of $X$, of dimensions $d_1\le\cdots\le d_s=\dim(X)$; the $X_i$'s are also Nash sets with monomial singularities, and have pure dimension. Now, we choose a finite semialgebraic stratification ${\mathcal G}$ of $M$ compatible with $X$ and the $X_i$'s. This can be done so that each stratum is locally connected at every adherent point in $M$ of it (see Preliminaries). Furthermore, by Proposition \ref{satype} we can suppose that the types of $X$ and all $X_i$'s are constant on every stratum. For, $\mathcal{G}$ can be chosen compatible with the semialgebraic sets defined by all types $\varLambda$ of the singularities of $X$ and the $X_i$'s.

\vspace{1mm}\setcounter{substep}{0}
\begin{substeps}{fgc}\label{fgc1step}\em
Let $\varGamma\in{\mathcal G}$ be a stratum of dimension $d_i$ contained in $X_i$. Then for each point $x\in\ol{\varGamma}\subset X_i$ the Nash closure $\ol{\varGamma}_x^{\rm an}$ of the germ $\varGamma_x$ is non-singular of dimension $d_i$; in fact, it is an irreducible component of $X_{i,x}$.
\end{substeps}

Indeed, given $x\in\ol{\varGamma}$ and since $X_i$ is a Nash set with monomial singularities of pure dimension $d_i$ at $x$, there exist non-singular germs $Y_{1,x},\dots, Y_{r_i,x}$, all of dimension $d_i$, such that
$$
\ol{\varGamma}_x^{\rm an}\subset X_{i,x}=Y_{1,x}\cup\cdots\cup Y_{r_i,x}.
$$
Since all germs $Y_{1,x},\dots, Y_{r_i,x},\,\ol{\varGamma}_x^{\rm an}$ are irreducible of the same dimension, there exists an index $k=1,\ldots,{r_i}$ such that $\ol{\varGamma}_x^{\rm an}=Y_{k,x}$, which implies (\ref{fgc}.\ref{fgc1step}).

Let us fix a stratum $\varSigma\in{\mathcal G}$ contained in $X$ and let us fix an irreducible component $X_i$ with $\varSigma \subset X_i$ (note here that $X_i$ contains $\varSigma$ if $X_i\cap\varSigma\ne\varnothing$). Since the type of each $X_i$ is constant on $\varSigma$ we have that $\mult(X_i,x)$ is a constant $r_i$ independent of $x\in \varSigma$ . Consider 
$$
{\mathcal G}_{\varSigma, i}=\{\varGamma\in{\mathcal G}:\varSigma\subset\ol{\varGamma}\subset X_i\}.
$$
We claim that

\vspace{1mm}
\begin{substeps}{fgc}\label{fgc2step}\em
There exist strata $\varGamma_{i1}, \ldots,\varGamma_{ir_i}\in{\mathcal G}_{\varSigma, i}$ such that for every $x\in\varSigma\subset X_i$
$$
X_{i,x}=\ol{\varGamma}_{i1,x}^{\rm an}\cup\cdots\cup\ol{\varGamma}_{ir_i,x}^{\rm an}
$$
is the decomposition of the analytic germ $X_{ix}$ into irreducible components. 
\end{substeps}

Indeed, the union $\varDelta=\bigcup_{\varGamma\in{\mathcal G}_{\varSigma,i}}\varGamma$ is a neighborhood of $\varSigma$ in $X_i$, so that $\varDelta_x=X_{i,x}$ for every $x\in\varSigma$. Now, fix a point $x_0\in\varSigma$. We have
$$
\bigcup_{\varGamma\in{\mathcal G}_{\varSigma,i}}\ol{\varGamma}^{\rm an}_{x_0}=\ol{\varDelta}^{\rm an}_{x_0}=X_{i,x_0}=A_{1,x_0}\cup\cdots\cup A_{r_i,x_0},
$$
where $A_{1,x_0},\ldots,A_{r_i,x_0}$ are distinct, non-singular germs at $x_0$ of dimension $d_i$. By (\ref{fgc}.\ref{fgc1step}), $A_{k,x_0}=\ol{\varGamma}_{ik,x_0}^{\rm an}$ for some $\varGamma_{ik}\in{\mathcal G}_{\varSigma,i}$ of dimension $d_i$. As far, the strata $\varGamma_{i1},\ldots,\varGamma_{ir_i}$ work only for the chosen point $x_0$, but we see readily that they work for all points in $\varSigma$.

Indeed, for any other point $x\in\varSigma\subset X_i$ we have 
$$
\ol{\varGamma}_{i\ell, x}^{\rm an}\subset X_{i,x}=Y_{1,x}\cup\cdots\cup Y_{r_i,x},
$$
where $Y_{1,x},\ldots,Y_{r_i,x}$ are distinct, non-singular germs of dimension $d_i$ at $x$. Thus, every $\ol{\varGamma}_{i\ell , x}^{\rm an}$ coincides with one of the $Y_{k,x}$'s, and what we must see is that $\ol{\varGamma}_{i\ell , x}^{\rm an}\neq\ol{\varGamma}_{i\ell' , x}^{\rm an}$ for $\ell\neq\ell'$. Consider the set $E$ defined as follows in two different ways:
$$\{x\in\varSigma:\ol{\varGamma}_{i\ell , x}^{\rm an}\neq\ol{\varGamma}_{i\ell' ,x}^{\rm an}\}=\{x\in\varSigma:\dim(\ol{\varGamma}_{i\ell , x}^{\rm an}\cap\ol{\varGamma}_{i\ell' ,x}^{\rm an})<d_i\}.
$$
Let us show that the first description implies that $E$ is closed in $\varSigma$, while the second implies that $E$ is open in $\varSigma$ (there is a subtlety here since we first consider germs and then analytic closures).

Consider $x\in\varSigma$. Since $\ol{\varGamma}_{i\ell , x}^{\rm an}$ is non-singular of dimension $d_i$, it is the germ at $x$ of a affine Nash manifold $D$ of dimension $d_i$, hence the point $x$ has a neighborhood $V^x$ such that 
$\varGamma_{i\ell}\cap V^x\subset D$, and so $\ol{\varGamma}_{i\ell , y}^{\rm an}\subset D_y$ for $y\in \varSigma\cap V^x$.
But $\ol{\varGamma}_{i\ell , y}^{\rm an}$ has dimension $d_i$, and $D_y$ is irreducible of that dimension, hence $\ol{\varGamma}_{i\ell , y}^{\rm an}= D_y$. Similarly, shrinking $V^x$ we find a affine Nash manifold $D'$ of dimension $d_i$ such that $\ol{\varGamma}_{i\ell' , y}^{\rm an}= D'_y$ for $y\in\varSigma\cap V^x$. Thus
$$
E\cap V^x=\{y\in\varSigma\cap V^x: D_y\neq D'_y\}=\{y\in\varSigma\cap V^x:\dim(D_y\cap D'_y)<d_i\}.
$$
Now the fact that $E$ is closed and open in $\varSigma$ is clearer.
 
Finally, since $x_0\in E$, the semialgebraic set $E$ is not empty, and since $\varSigma$ is connected, we conclude that $E=\varSigma$, as desired. The claim (\ref{fgc}.\ref{fgc2step}) is proved.

To complete the argument, consider the locally compact semialgebraic set $T_{ik}=\ol{\varGamma}_{ik}\supset\varSigma$ for $k=1,\dots, r_i$. Since 
$\ol{T}_{ik ,x}^{\rm an}=\ol{\varGamma}_{ik ,x}^{\rm an}$ for each $x\in T_{ik}$, condition (\ref{fgc}.\ref{fgc1step}) implies that 
there exists a connected affine Nash manifold $S_{ik}\subset M$ of dimension $d_i$, containing $T_{ik}$, 
hence $\varSigma$, with $S_{ik ,x}=\ol{T}_{ik ,x}^{\rm an}$ for $x\in\varSigma$ (see Fact \ref{closedinmani}). 

Thus:
$$
X_{i,x}=\ol{\varGamma}_{i1,x}^{\rm an}\cup\cdots\cup\ol{\varGamma}_{ir_i,x}^{\rm an}=
\ol{T}_{i1,x}^{\rm an}\cup\cdots\cup\ol{T}_{ir_i,x}^{\rm an}=
S_{i1,x}\cup\cdots\cup S_{ir_i,x},
$$
and
\begin{equation}\label{eq:equalgerm}
X_x={\bigcup}_iX_{i,x}={\bigcup}_i\big(S_{i1,x}\cup\cdots\cup S_{ir_i,x}\big)
\end{equation}
for all $x\in\varSigma$. It follows that $\varSigma$ is contained in  $\text{Int}_{X}\big({\bigcup}_i(S_{i1}\cup\cdots\cup S_{ir_i})\big)$, which is an open semialgebraic subset of $X$. Therefore, there exists an open semialgebraic neighborhood $U$ of $\varSigma$ in $M$ such that
$$
X\cap U={\bigcup}_i\big(S_{i1}\cup\cdots\cup S_{ir_i}\big) \cap U.
$$ 
Moreover, since each $S_{ik}$ is locally closed we can assume that $S_{ik}$ is closed in $U$. Finally note that 
$$
\varSigma\subset {\bigcap}_i\big(S_{i1}\cap\cdots\cap S_{ir_i}\big).
$$
Thus, we end by renaming the $S_{ik}\cap U$'s as the $Y_j$'s of the statement. Concerning the last assertion there, we only need to remark that no irreducible component of a germ $X_{i,x}$ is contained in one of $X_{j,x}$, because $X_i$ and $X_j$ are irreducible of pure dimension. 
\end{proof}

\begin{stepsp}{ \em Proof of Theorem \em\ref{thmfiniteness}.}\label{pfthmfiniteness}
By Proposition \ref{fgc} we are reduced to the following claim.

\vspace{1mm}\setcounter{substep}{0}
\begin{substeps}{pfthmfiniteness}\em
Let $X\subset M$ be a Nash set with monomial singularities that is a union of closed Nash manifolds $Y_1,\dots, Y_s$. Let $\varSigma\subset Y_1\cap\cdots\cap Y_s$ be a connected Nash manifold such that: \em(i) \em the type of $X$ is constant on $\varSigma$ and \em (ii) \em $Y_{1,x},\dots,Y_{s,x}$ are the distinct irreducible components of the germ $X_x$ for every $x\in\varSigma$. Then $\varSigma$ can be covered by finitely many open semialgebraic sets $U$ equipped with Nash diffeomorphisms $u: U\to\R^m$ such that $u(X\cap U)$ is a union of coordinate linear varieties.
\end{substeps}

To start with, we refine the fact that the type $\varLambda$ is constant on $\varSigma$. We show that the load of $\{T_xY_1,\ldots,T_xY_s\}$ is the same for all $x\in \varSigma$. Since $\varSigma$ is connected it is enough to show that the load is locally constant. Fix $x\in \varSigma$ and let $u:U\mapsto \R^m$ be a local Nash diffeomorphism with $u(x)=0$ and $u(X\cap U)=\bigcup_{\lambda\in \varLambda}\{u_{\lambda}=0\}$. By (ii), for $U$ small enough, the irreducible components of $X\cap U$ are $Y_1\cap U,\ldots,Y_s\cap U$ and we denote $L_i=u(Y_i\cap U)$. For any point $y \in  \varSigma\cap U \subset (Y_1\cap \cdots \cap Y_s) \cap U= u^{-1}(L_1\cap \cdots \cap L_s)$ we have $d_y(T_yY_i)=L_i$, and consequently the loads of $\{T_yY_1,\ldots,T_yY_s\}$ and $\{T_xY_1,\ldots,T_xY_s\}$ are the same. Thus, the load is locally constant in $\varSigma$ and therefore constant. 

After this, fix any family $\mathcal{L}=\{L_1,\ldots,L_s\}$ whose type is $\varLambda$ and whose load is that of $\{T_xY_1,\ldots,T_xY_s\}$ for all $x\in \varSigma$.

Now, $X$ being a Nash set with monomial singularities, we can assume that every intersection $Y_{i_1}\cap\cdots\cap Y_{i_r}$ is an affine Nash manifold, and at every $x\in Y_1\cap\cdots\cap Y_s$ we have
$$
T_x(Y_{i_1}\cap\cdots\cap Y_{i_r})=T_xY_{i_1}\cap\cdots\cap T_xY_{i_r},
$$
or equivalently $\dim_x(Y_{i_1}\cap\cdots\cap Y_{i_r})=\dim(L_{i_1}\cap\cdots\cap L_{i_r})$. Indeed, since $\varSigma$ is connected we can assume that $Y_1,\ldots,Y_s$ are connected. Then, by the Identity Principle $Y_1,\ldots,Y_s$ are the irreducible components of $X$ because by hypothesis their germs at any point $x\in \varSigma$ are the distinct irreducible components of $X_x$. Moreover, again by the Identity Principle, for any $x\in X$ the irreducible components of $X_x$ are the $Y_{i,x}$ with $x\in Y_i$. Thus, and since $X$ has monomial singularities, for any $x\in X$ the germ $Y_{i_1,x}\cup \cdots\cup Y_{i_r,x}$ is Nash diffeomorphic to a union of germs of coordinate varieties. In particular, $Y_{i_1,x}\cap \cdots\cap Y_{i_r,x}$ is Nash diffeomorphic to a germ of a coordinate variety and the equality of dimensions above is straightforward.

After this preparation, the union $\bigcup_iL_i$ is the model to which $X=\bigcup_i Y_i$ must be diffeomorphic \em near $\varSigma$ in the local semialgebraic sense\em. To confirm this, we work up the levels 
$$\bl^{(p)}=\big\{L_I=\bigcap_{i\in I}L_i:\#I=p\big\}$$ and the corresponding levels $$\mathcal{Y}^{(p)}=\big\{Y_I=\bigcap_{i\in I}Y_i:\#I=p\big\}$$  for $p=1,\ldots,s$. For the moment being we do not look for Nash diffeomorphisms whose images are $\R^m$, they will be just open semialgebraic subsets of $\R^m$. We will fix this later.

\em Step 1\em. At bottom level $p=s$ we have one single intersection $Y_1\cap\cdots\cap Y_s$ which is a Nash manifold of the same dimension that the linear coordinate variety $L_1\cap\cdots\cap L_s$. Thus, $Y_1\cap\cdots\cap Y_s$, hence $\varSigma$, can be covered by finitely many open semialgebraic sets $W$ such that the intersection $U_s=W\cap Y_1\cap\cdots\cap Y_s$ is Nash diffeomorphic to an open semialgebraic set $\varOmega_s$ of $L_1\cap\cdots\cap L_s$. Since there are finitely many $W$'s, we can suppose simply that $\varSigma\subset U_s$ and we have a Nash diffeomorphism $\varphi_s:\varOmega_s\to U_s$. Denote $S=\varphi_s^{-1}(\varSigma)$.

\em Step 2\em. Now we formulate carefully the recursion procedure. Suppose we have, at level $p$, open semialgebraic subsets $\varOmega_p$ and $U_p$ of $\bigcup_{\#I=p}L_I$ and $\bigcup_{\#I=p}Y_I$ respectively, with $\varSigma\subset U_p$, and a continuous semialgebraic map $\varphi_p:\varOmega_p\to U_p$ which restricts to Nash diffeomorphisms $L_K\cap\varOmega_p\to Y_K\cap U_p$; here $K$ stands for any subset of $\{1,\ldots,s\}$ with at least $p$ indices. We also assume that $\varOmega_k\subset \varOmega_p$ and $\varphi_p|_{\varOmega_k}=\varphi_k$ for $k=p+1,\ldots,s$. We want an extension to a similar map $\varphi_{p-1}$ at level $p-1$, after some suitable finite semialgebraic partition of $\varSigma$. We proceed as follows.

\em Step 3\em. Let $Y_J$ be an intersection of $Y_i$'s at level $p-1$ and consider the continuous semialgebraic restriction $\phi_J$ of $\varphi_p$ to $L_J\cap\varOmega_p$. Moreover, for any set $I$ of $p$ indices the restriction induces a Nash diffeomorphism $L_J\cap L_I\cap\varOmega_p\to Y_J\cap Y_I\cap U_p$ (note that every intersection of $Y_J$ with another $Y_{J'}$ at  
the same level $p-1$ is included in some intersection $Y_I$ at level $p$). By considering a finite open semialgebraic covering of $Y_J$ we can suppose it is Nash 
diffeomorphic to an affine space, and we pick any identification $Y_J\equiv\R^a$. Of course $a=\dim(L_J)$. 

Now, for any subset $I$ of $p$ indices the tangent bundle of the affine Nash manifold $Y_J\cap Y_I\subset\R^a$ is generated by its global Nash sections \cite[12.7.10]{bcr}. Since $U_s$ is contained in all those $Y_J\cap Y_I$'s, the sums
$$
\sum_{\#I=p}T_x(Y_J\cap Y_I)\subset\R^a\equiv L_J\qquad\text{for $x\in U_s$}
$$
are the fibers of a Nash vector bundle $\sT$ on $U_s$ also generated by its global Nash sections. The same is true for the orthogonal bundle $\sT^\perp$, say it is generated by finitely many global Nash sections $\zeta_i:U_s\to\R^a$. 
For every $x\in\varSigma$ the codimension $c$ of
$\sum_{\#I=p}T_x(Y_J\cap Y_I)$ in $T_xY_J$ is that of $\sum_{\#I=p}L_J\cap L_I$ in $L_J$ (since the load of $\bl=\{L_1,\ldots,L_s\}$ and $\{T_xY_1,\ldots,T_xY_s\}$ are the same). Therefore, for every $x\in\varSigma$, we have that $c$ of the $\zeta_i(x)$'s are independent. After considering once again a finite open covering of $\varSigma$, we can assume we have exactly $c$ independent Nash maps $\zeta_1,\dots,\zeta_c:U_s\to\R^a\equiv Y_J$ such that $\zeta_1(x),\dots,\zeta_c(x)$ span $\sT_x^\perp$ for every $x\in\varSigma$. In other words, 
$$
\R^a={\sum}_{\#I=p}T_x(Y_J\cap Y_I)+L[\zeta_1(x),\dots,\zeta_c(x)].
$$
for $x\in\varSigma$.

Next consider $V_J=\sum_{\#I=p}L_J\cap L_I$ and notice that the orthogonal supplement $W_J$ of $V_J$ in $L_J\subset\R^m$ is a coordinate linear variety of dimension $c$ generated by suitable vectors $e_{\ell_1},\dots,e_{\ell_c}$ of the canonical basis. Let $\varOmega'=\varOmega_s\oplus W_J$ which is an open semialgebraic subset of the coordinate linear variety $L'=(L_1\cap\cdots\cap L_s)\oplus W_J$. We extend $\varphi_s$ as follows:
$$
\varphi^*_s(v+\alpha_1e_{\ell_1}+\cdots+\alpha_ce_{\ell_c})
 =\varphi_s(v)+\alpha_1\zeta_1(\varphi_s(v))+\cdots+\alpha_c\zeta_c(\varphi_s(v)).
$$
This extension $\varphi^*_s$ is defined on $\varOmega'$. Since $\varOmega'\cap(L_J\cap \varOmega_p)=\varOmega_s$ and $\phi_J|_{\varOmega_s}=\varphi_s$, the maps $\varphi^*_s$ and $\phi_J$ glue into
a continuous semialgebraic map $\phi_J':\varOmega'\cup(L_J\cap \varOmega_p)\to\R^a\equiv Y_J$. Clearly, the components of $\phi_J'$ are $\tt c$-Nash functions on the Nash sets with monomial singularities $\varOmega'\cup(L_J\cap \varOmega_p)$. Consequently, they are Nash functions, and $\phi_J'$ has a Nash extension $\ol\phi_J:\varOmega_J\to\R^a\equiv Y_J$ to some open semialgebraic neighborhood $\varOmega_J$ of $\varOmega'\cup(L_J\cap \varOmega_p)$ in $L_J$. We claim that \em $\ol\phi_J$ is a local diffeomorphism at every $v\in S=\varphi_s^{-1}(\varSigma)$\em.

Indeed, fix such a $v\in S$ and put $x=\varphi_s(v)$. Since $\ol\phi_J$ is a diffeomorphism on every $L_J\cap L_I\cap\varOmega_p$ with $\#I=p$ we have 
 $d_v\ol\phi_J(L_J\cap L_I)=T_x(Y_J\cap Y_I)$
and so $d_v\ol\phi_J(L_J)\subset\R^a$ contains the sum $\sum_{\#I=p}T_x(Y_J\cap Y_I)$. Moreover, for every $e_{\ell_k}$ we have
$$
\ol\phi_J(v+te_{\ell_k})=x+t\zeta_k(x), \ t\in\R.
$$
Hence, $d_v\ol\phi_J(e_{\ell_k})=\zeta_k(x)$. Consequently, 
$$
\R^a={\sum}_{\#I=p}T_x(Y_J\cap Y_I)+L[\zeta_1(x),\dots,\zeta_c(x)]\subset d_v\ol\phi_J(L_J),
$$
and $d_v\ol\phi_J:L_J\to\R^a\equiv Y_J$ is onto, hence a linear isomorphism (as $a=\dim(L_J)$). Our claim is proved.

Consequently, shrinking $\varOmega_J$, $\ol\phi_J$ is a local Nash diffeomorphism onto an open semialgebraic 
neighborhood $U_J$ of $\varSigma$ in $Y_J$. In this situation there is a finite open semialgebraic covering $\{\varOmega_{Jk}\}$ of 
$\varOmega_J$ such that every restriction $\ol\phi_J|_{\varOmega_{Jk}}$ is a Nash diffeomorphism onto an 
open semialgebraic set $U_{Jk}$ of $Y_J$ (see \cite[9.3.9, p.\hspace{1.5pt}226]{bcr}). As usual, this means that we can suppose $\ol\phi_J:\varOmega_J\to U_J$ is a diffeomorphism.

By construction, the semialgebraic set $\bigcup_J\varOmega_J$ is a neighborhood of 
$S$ in $\bigcup_JL_J$, hence it contains an open semialgebraic neighborhood 
$\varOmega_{p-1}$. Notice that
the diffeomorphisms $\ol\phi_J$ glue together to give a continuous semialgebraic extension 
$\varphi_{p-1}:\varOmega_{p-1}\to U_{p-1}$ of $\varphi_p$ that verifies all conditions required.

\em Step 4\em. Thus climbing from level to level, we get a continuous semialgebraic map $\varphi_1:\varOmega_1\to U_1$ from an open semialgebraic neighborhood $\varOmega_1$ of $S$ in $L_1\cup\cdots\cup L_s$ into another $U_1$ of $\varSigma$ in $X=Y_1\cup\cdots\cup Y_s$, which induces Nash diffeomorphisms $L_i\cap\varOmega_1\to Y_i\cap W_1$. Then we apply again the 
same argument above to extend $\varphi_1$ to a Nash diffeomorphism $\varphi:\varOmega\to U$ from an 
open semialgebraic neighborhood $\varOmega$ of $S$ in $\R^m$ onto one $U$ of $\varSigma$ in $M$. (Of course in climbing 
we have used many finite semialgebraic coverings and shrunk the neighborhood of $\varSigma$, so that what we 
really have is a finite collection of such maps whose images $U$ cover $\varSigma$.) 

\em Final arrangement\em. In this situation, $\varphi^{-1}:U\to\varOmega$ is the diffeomorphism we were looking for, except that $\varOmega$ need not be $\R^m$.
To amend this, notice that $S=\varphi^{-1}(\varSigma)\subset L_1\cup\cdots\cup L_s$ is contained in the 
intersection $L_1\cap\cdots\cap L_s$. Now this intersection can be written as the intersection $H_{j_1}\cap \cdots \cap H_{j_q}$ of the coordinate hyperplanes $H_j$ that contain some $L_i$ (that is, we just collect the equations $x_{j_\ell}=0$ of all $L_i$'s). This said, we know that $\varOmega$ 
contains a smaller neighborhood of $S$, which is a finite union of open semialgebraic sets Nash 
diffeomorphic to $\R^m$, by diffeomorphisms that preserve the coordinate hyperplanes $H_{j_\ell}$, hence the coordinate linear varieties $L_i$ (see \cite[4.4.5]{fgr}). Composing $\varphi^{-1}$ 
with the latter we obtain the diffeomorphism $u$ we sought.
\fin
\end{stepsp}

The preceding construction gives a characterization of monomial singularities as announced before.

\begin{cor}\label{cor:sufficient}
Let $X_x$ be a Nash germ whose irreducible components $X_{1,x},\ldots,X_{s,x}$ are non-singular. Then $X_x$ is a monomial singularity if and only if the tangent cone $\{T_xX_1,\ldots,T_xX_s\}$ is a extremal family and 
$$
\dim(X_{i_1,x}\cap \cdots \cap X_{i_r,x})=\dim(T_xX_{i_1}\cap \cdots \cap T_xX_{i_r})
$$
for $1\le i_1<\cdots<i_r\le s$.
\end{cor}
\begin{proof}The only if part is clear. For the converse implication, we have to prove first that each intersection $X_{i_1,x}\cap \cdots \cap X_{i_r,x}$ is non-singular. Once this is done, by hypothesis
$$
T_x(X_{i_1,x}\cap \cdots \cap X_{i_r,x})=T_xX_{i_1}\cap \cdots \cap T_xX_{i_r}
$$
which enables us to repeat the preceding proof \ref{pfthmfiniteness}. 

Using the obvious induction argument, to show that each intersection $X_{i_1,x}\cap \cdots \cap X_{i_r,x}$ is non-singular, we are reduced to prove that: \em If $X_{1,x},X_{2,x}\subset\R^a_x$ are two non-singular germs such that
$\dim(X_{1,x}\cap X_{2,x})=\dim(T_xX_{1,x}\cap T_xX_{2,x})$, then $X_{1,x}\cap X_{2,x}$ is non-singular\em.

Indeed, as it is well-known, the orthogonal projection $\pi_j:\R^a\to L_j$ induces a Nash diffeomorphism between $X_{j,x}$ and $L_{j,x}$. Consider next the orthogonal projections $\pi:\R^a\to L=L_1+L_2$ and $\rho_j:L\to L_j$. Since $\pi_j=\rho_j\circ\pi$ it holds that $X'_{j,x}=\pi(X_{j,x})\subset L_x$ is Nash diffeomorphic to $X_{j,x}$ and in particular $\pi(X_{1,x}\cap X_{2,x})$ is a Nash subset germ of each $X_{j,x}$ Nash diffeomorphic to $X_{1,x}\cap X_{2,x}$; hence, $\dim(\pi(X_{1,x}\cap X_{2,x}))=\dim(L_1\cap L_2)$. On the other hand, since $X'_{1,x}$ and $X'_{2,x}$ in $L_x$ are transversal, $X'_{1,x}\cap X'_{2,x}$ is the germ of a Nash manifold of dimension $\dim(L_1\cap L_2)$. As $\pi(X_{1,x}\cap X_{2,x})$ is a Nash subset germ of $X'_{1,x}\cap X'_{2,x}$ of its same dimension, and the latter is irreducible, we deduce that they are equal and therefore $X_{1,x}\cap X_{2,x}$ is non-singular, as required.\end{proof}

We cannot drop the condition on the dimension of the intersections and the tangent spaces. In $\R^4$, let $X_1=\{x_3=x_4=0\}$, $X_2=\{x_2=0,x_4=x_1^2\}$ and let $X$ be their union. The tangent cone of $X$ at the origin is $\{x_3=0,x_4=0\},\{x_2=0,x_4=0\}$ and therefore it is extremal. However, the intersection $X_1\cap X_2$ is the origin.

\section{Extension linear maps for Nash sets with monomial singularities}\label{sec:extension}

This section is preliminary for Nash approximation. As explained in \ref{topfunc}, for any Nash set $X$ the ring $\sS^\nu(X)$ of $\sS^\nu$-functions is equipped with the topology as the quotient $\sS^\nu(X)=\sS^\nu(M)/I^\nu(X)$. If $X_1,\ldots,X_s$ are the irreducible components of $X$ then ${}^{\tt c}\sS^\nu(X)$ is equipped with the topology induced by the inclusion ${}^{\tt c}\sS^\nu(X)\to\sS^\nu(X_1)\times \cdots \times\sS^\nu(X_s)$ given by the multiple restriction $f\mapsto (f|_{X_1},\ldots,f|_{X_s})$. Moreover, the inclusion $\gamma:\sS^\nu(X)\to {}^{\tt c}\sS^\nu(X)$ is always continuous. Here our purpose is to prove that if $X$ is a Nash set with monomial singularities then there exist a continuous extension linear map $\theta:\sS^\nu(X)\to\sS^\nu(M)$ and that $\gamma$ is a homeomorphism (Proposition \ref{topirre}). Both statements will be deduced from the existence of a continuous extension linear map ${}^{\tt c}\theta:{}^{\tt c}\sS^\nu(X)\to\sS^\nu(M)$. Recall that in Proposition \ref{trick} we already found such an extension for 
unions of coordinate linear varieties. However the continuity of the extension trick there may fail because composition on the right with the orthogonal projections need not be continuous \cite[II.1.5, p.\hspace{1.5pt}83]{s}. Our purpose now is to amend this.

\begin{lem}\label{lem:localextcont}
Let $X=L_1\cup\cdots\cup L_s$ be a union of coordinate linear varieties in $\R^m$. Then there is a continuous extension linear map ${}^{\tt c}\theta:{}^{\tt c}\sS^\nu(X)\to\sS^\nu(\R^m)$. 
\end{lem}
\begin{proof}
Fix a non-empty set of indices $I=\{i_1,\dots,i_r\}\subset \{1,\ldots,s\}$ and set $L_I=\bigcap_{i\in I}L_i$ and $X^I=\bigcup_{i\in I}L_i$. By Proposition \ref{trick} every ${}^{\tt c}\sS^\nu$ function $h:X^I\to\R$ has the following $\sS^\nu$ extension to $\R^m$
$$
H_I=\!\sum_{\varnothing\ne J\subset I}(-1)^{\#(I)+1}h\circ\pi_J.
$$
Now consider
the open semialgebraic set
$$
\varOmega_I=\{x\in\R^m:\dist(x,L_I)<1)\}\setminus\bigcup_{j\notin I}L_j,
$$
which satisfies $X^I\cap\varOmega_I=X\cap\varOmega_I$. We claim that the extension linear map
$$
{}^{\tt c}\theta_I:{}^{\tt c}\sS^\nu(X^I)\to\sS^\nu(\varOmega_I):h\mapsto H_I |_{\varOmega_I}
$$
is continuous. Note that this is our previous extension linear map followed by a restriction that makes it continuous.

Indeed, it is enough to show that $h\mapsto h\circ\pi_J|_{\varOmega_I}$ is continuous. Here we consider the topology defined in ${}^{\tt c}\sS^\nu(X^I)$ as subset of $\sS^\nu(L_{i_1})\times\cdots\times\sS^\nu(L_{i_r})$ (this is the reason to keep referring to semialgebraic ${}^{\tt c}\sS^\nu$ functions, although we already know they are all $\sS^\nu$ functions). Clearly, it is enough to see that if all restrictions $h|_{L_i},\,i\in I,$ are close enough to zero, then $h\circ\pi_J|_{\varOmega_I}$ is arbitrarily close to zero. Thus, pick any positive continuous semialgebraic function $\veps:\varOmega_I\to\R$. We know from \L ojasiewicz's inequality \cite[2.6.2]{bcr} that there are a constant $C>0$ and an integer $p$ large enough so that
$$
\frac{1}{(C+\|x\|^2)^p}<\veps(x)\quad\text{for every $x\in\varOmega_I$}.
$$
Let $x\in\varOmega_I$ and $J\subset I$; in particular, $L_I\subset L_J$. Then for the orthogonal projection $\pi_J:\R^m\to L_J$ we have
$$
\|x\|^2=\dist(x,L_J)^2+\|\pi_J(x)\|^2\le\dist(x,L_I)^2+\|\pi_J(x)\|^2<1+\|\pi_J(x)\|^2
$$
and so
$$
\frac{1}{(C+1+\|\pi_J(x)\|^2)^p}<\frac{1}{(C+\|x\|^2)^p}<\veps(x).
$$
Denote $\delta(x)=\frac{1}{(C+1+\|x\|^2)^p}$ and suppose that all restrictions $h|_{L_i},\,i\in I,$ are $\delta$ close to zero in the $\sS^\nu$ topology. Let us check that $h\circ\pi_J$ is $\veps$ close to zero. Look at any partial derivative
$$
\frac{\partial^{|\alpha|}(h\circ\pi_J)}{\partial x_1^{\alpha_1}\cdots\partial x_m^{\alpha_m}}(x),\quad \text{ where } |\alpha|=\alpha_1+\cdots+\alpha_m\leq\nu,
$$
at a point $x\in\varOmega_I$. Since composition with $\pi_J$ is substituting zero for the coordinates in $L_J$, we see that $h\circ\pi_J$ does not depend on those coordinates, which implies that the above partial derivative is zero whenever such a coordinate appears in the derivative. Thus we look at derivatives that do not include them. But for those, we have
$$
\Big|\frac{\partial^{|\alpha|}(h\circ\pi_J)}{\partial x_1^{\alpha_1}\cdots\partial x_m^{\alpha_m}}(x)\Big|
 =\Big|\frac{\partial^{|\alpha|}h}{\partial x_1^{\alpha_1}\cdots\partial x_m^{\alpha_m}}(\pi_J(x))\Big|
 <\delta(\pi_J(x))<\veps(x)
$$
because $\pi_J(x)\in L_J\subset L_i$ for some $i\in I$. Hence ${}^{\tt c}\theta_I$ is continuous, as required.

Finally, we glue the ${}^{\tt c}\theta_I$'s. Consider a semialgebraic $\sS^\nu$ partition of unity $\{\phi,\phi_I:I\}$ subordinated to 
$\{\R^m\setminus X,\varOmega_I:I\}$, which is an open semialgebraic covering of $\R^m$. Define
$$
{}^{\tt c}\theta:{}^{\tt c}\sS^\nu(X)\to\sS^\nu(\R^m):h\mapsto {\sum}_I\phi_I\cdot{}^{\tt c}\theta_I(h|_{X^I})
$$
where each $\phi_I\cdot{}^{\tt c}\theta_I(h|_{X^I})$ extends by zero off $\varOmega_I$. Since $\phi$ vanishes on $X$, $\sum_I\phi_I\equiv1$ on $X$; hence, ${}^{\tt c}\theta(h)$ is a semialgebraic $\sS^\nu$ extension of $h$. Finally, ${}^{\tt c}\theta$ is continuous because so are all ${}^{\tt c}\theta_I$'s, as required.
\end{proof}

The preceding construction can be extended by finiteness to an arbitrary Nash set with monomial singularities.

\begin{prop}\label{topirre}
Let $X\subset M$ be a Nash set with monomial singularities. Then there are continuous linear maps $\theta:\sS^\nu(X)\to\sS^\nu(M)$ such that $\theta(h)|_X=h$. Moreover, if $X_1,\dots,X_s$ are the irreducible components of $X$, then the multiple restriction homomorphism $\sS^\nu(X)\to\sS^\nu(X_1)\times\cdots\times\sS^\nu(X_s)$ is a closed embedding.
\end{prop}
\begin{proof}
Using a partition of unity and the preceding local result we obtain a continuous extension linear map 
${}^{\tt c}\theta:{}^{\tt c}\sS^\nu(X)\to\sS^\nu(M)$. Consider the commutative diagram 
\setlength{\unitlength}{1mm}
\begin{center}
\begin{picture}(60,15)
\small
\put(0,13){$\sS^\nu(X)\overset{\gamma}{\longrightarrow}{}^{\tt c}\sS^\nu(X)\hookrightarrow\sS^\nu(X_1)\times\cdots\times\sS^\nu(X_s)$}
\put(7.5,10.5){\vector(1,-1){8}}\put(17,3){\vector(-1,1){8}}\put(23,11){\vector(0,-1){7}}
\put(7,4){$\theta$}
\put(14,8){$\rho$}\put(24,7){${}^{\tt c}\theta$}
\put(18.5,0){$\sS^\nu(M)$}
\end{picture}
\end{center}
where $\theta={}^{\tt c}\theta\circ\gamma$ and $\rho:\sS^\nu(M)\to\sS^\nu(X)$ is the continuous restriction epimorphism. All maps here are continuous and therefore $\theta={}^{\tt c}\theta\circ\gamma$ is a continuous extension linear map. On the other hand, $\gamma^{-1}=\rho\circ{}^{\tt c}\theta$ and so it is continuous, that is, $\gamma$ is a homeomorphism. In \ref{topfunc} we proved that the inclusion of $\sS^\nu(X)$ in $\sS^\nu(X_1)\times\cdots\times\sS^\nu(X_s)$ is closed and therefore the multiple restriction is a closed embedding.
\end{proof}

\section{Approximation for functions}\label{section:Aproxfunctions}

We now discuss approximation for Nash sets with monomial singularities. By Proposition \ref{prop:absapproxfun} absolute approximation is possible in a general situation. However, we are interested for the applications in a stronger relative version which in Proposition \ref{closext} we prove for Nash sets with monomial singularities. Let us see first some useful facts that relate the zero-ideal of a Nash set and this relative approximation.

\begin{lem}\label{lem:equivzeroapprox}
Let $X$ be a Nash set of a Nash manifold $M\subset \R^a$ of dimension $m$ and let $\nu\geq m$. If $I^\nu(X)\subset I(X)\sS^{\nu-m}(M)$ then every semialgebraic $\sS^\nu$ function $F:M\to \R$ whose restriction to $X$ is Nash can be $\sS^{\nu-m}$ approximated by Nash functions $H:M\to \R$ that coincide with $F$ on $X$.
\end{lem}
\begin{proof}Suppose that $I^\nu(X)\subset I(X)\sS^{\nu-m}(M)$. Since $F|_X$ is Nash there is some Nash function $G:M\to\R$ with $G|_X=F|_X$. Therefore $F-G$ is a semialgebraic $\sS^\nu$ function vanishing on $X$, so that $F-G=\sum\psi_if_i$ for some Nash functions $f_i\in I(X)$ and some $\sS^{\nu-m}$ functions $\psi_i$. But by absolute approximation (Fact \ref{shiota}), we find Nash functions $\varphi_i$ that are $\sS^{\nu-m}$ close to $\psi_i$, and then $H=G+ \sum\varphi_if_i$ is a Nash function $\sS^{\nu-m}$ close to $F$ such that $H|_X=F|_X$.
\end{proof}

Under coherence, it is enough to control ideals for a finite covering. 

\begin{lem}\label{locpi}
Let $X\subset M$ be a coherent Nash set. Suppose there is a finite semialgebraic open covering $X\subset U_1\cup\cdots\cup U_s$ such that 
$I^\nu(X\cap U_i)\subset I(X\cap U_i)\sS^{\nu-m}(U_i)$ for $1\le i\le s$. Then $I^\nu(X)\subset I(X)\sS^{\nu-m}(M)$.
\end{lem}
\begin{proof}
Adding $M\setminus X$ to the covering we can assume $M=\bigcup_i U_i$. Let $\{\varphi_i\}_i$ be an $\sS^\nu$ partition of unity subordinated to the $U_i$'s; recall that $\ol{\{\varphi_i>0\}}\subset U_i$. Since $X$ is coherent, $I(X)$ generates $I(U_i\cap X)$ (see equation \eqref{eq:ix} after Fact \ref{fact:cohananash}). We also know that $I(X)$ is finitely generated, say by $f_1,\dots,f_p$. Now let $f\in\sS^\nu(M)$ vanish on $X$. It follows that $f|_{U_i}$ vanishes on $U_i\cap X$ and, by hypothesis, there are Nash functions $g_{ik}\in \sS^{\nu-m}(U_i)$ such that
$$
f|_{U_i}=\sum_kg_{ik}(f_k|_{U_i}).
$$
Now consider the functions $\varphi_ig_{ik}$. Although defined on $U_i$, they vanish off $\ol{\{\varphi_i>0\}}\subset U_i$, hence can be extended by zero off $U_i$. Thus we have in fact $\varphi_ig_{ik}\in\sS^{\nu-m}(M)$. Finally one readily checks that
$$
f={\sum}_{k}\big(\,{\sum}_i \varphi_ig_{ik}\big)f_k\in I(X)\sS^{\nu-m}(M), 
$$
which concludes the proof.
\end{proof}

These lemmas reduce the relative approximation problem for Nash sets with monomial singularities to the study of the zero-ideal of a union of coordinate linear varieties.

\begin{prop}\label{genidsol}
Let $X\subset\R^m$ be a union of coordinate linear varieties $L_1,\ldots,L_s$ and let $I(X)$ be the ideal of Nash functions vanishing on $X$. Denote by $x^\sigma$ the square-free monomials associated to $X$ (see \em Definition \ref{def:squarefree}\em) and let $\#(x^\sigma)$ be the number of variables in $x^\sigma$. Let $\nu\ge\max_\sigma\#(x^\sigma)$ and let $f:\R^m\to\R$ be an $\sS^\nu$ function vanishing on $X$. Then 
$$
f={\sum}_\sigma f_\sigma x^\sigma,
$$
where $f_\sigma$ is an $\sS^{\nu-\#(x^\sigma)}$ function for each $\sigma$.
\end{prop}

\begin{remarks}\label{remgenidsol}Note that, since the monomials $x^\sigma$ are square-free, none has more than $m$ variables, that is $\#(x^\sigma)\le m$. Consequently, $I^\nu(X)\subset I(X)\sS^{\nu-m}(M)$. 
\end{remarks}

The proof of Proposition \ref{genidsol} consists of a double induction on the dimension $m$ and the number of linear varieties involved. We will use the following lemma:

\begin{lem}\label{easy}
Let $f$ be an $\sS^\nu$ function on $\R^m$. Write $x=(x_1,x')$ the variables in $\R^m$. Then there is an expansion
$$
f(x)=f_1(x)x_1+f(0,x'),
$$
where $f_1$ is an $\sS^{\nu-1}$ function.
\end{lem}
\begin{proof}
We know this from elementary Analysis with
$$
f_1(x)=\int_0^1\frac{\partial}{\partial x_1}f(tx_1,x')dt,
$$
but an integral can well not be semialgebraic; however, this particular one is semialgebraic. Indeed, the function $g(x)=(f(x)-f(0,x'))/x_1$ is semialgebraic in $x_1\ne0$, and the graph of $f_1$ is the closure of the graph of $g$. Hence $f_1$ is in fact semialgebraic.
\end{proof}

\begin{proof}[Proof of Proposition \em\ref{genidsol}]
We argue by induction on the dimension. If $m=1$ then $X=\R$ or $X=\{0\}$; in the first case it is obvious and in the second one it follows from Lemma \ref{easy} with $m=1$ and $f(0)=0$. Hence we assume $m>1$ and the result proved for dimension $m-1$. Let $X\subset\R^m$ be a union of coordinate linear varieties $L_1,\dots,L_s$ of $\R^m$ and let $f$ be an $\sS^\nu$ function vanishing on $X$. We argue again by induction, now on the number $s$ of $L_i$'s. 

The case $s=1$. We have $X=L_1$ and can suppose the equations of $L_1$ are $x_1=\cdots=x_r=0$. Lemma \ref{easy} gives $f=f_1x_1+g_1$, where $f_1$ is $\sS^{\nu-1}$ and $g_1$ is $\sS^\nu$ and does not depend on $x_1$. Again by the lemma, $g_1=f_2x_2+g_2$, where $f_2$ is $\sS^{\nu-1}$ and $g_2$ is $\sS^\nu$ and does not depend on either $x_1$ or $x_2$. And so on, till we have:
$$
f=f_1x_1+\cdots+f_rx_r+g,
$$
where all $f_k$'s are $\sS^{\nu-1}$ and $g$ does not depend on either of $x_1,\dots,x_r$. Thus, substituting $0$ for these variables does not affect $g$, and since $f$ vanishes on $x_1=\cdots=x_r=0$ we get $g\equiv0$. In other words, $f=f_1x_1+\cdots+f_rx_r$ and we are done.

Let now $s>1$ and assume the result known for less than $s$ linear varieties in $\R^m$.

We can suppose that $x_1$ is a variable of $L_1$ (see \ref{def:squarefree} for the precise terminology). Denote as usual $x=(x_1,x')$ the variables in $\R^m$. By Lemma \ref{easy} once again we have an $\sS^{\nu-1}$ function $f_1$ such that
$$
f(x)=f_1(x)x_1+g(x').
$$

We claim that $g$ vanishes on the $L_i$'s that do not contain the variable $x_1$. Indeed for such a coordinate linear variety $L_i$, $z=(z_1,z')\in L_i$ if and only if $(0,z')\in L_i$. Since $f$ vanishes on $L_i$ we have
$$
0=f(0,z')=f_1(0,z')0+g(z')=g(z')=g(z),
$$
which implies that $f(x)=f_1(x)x_1$ on $L_i$. Consequently, since $f|_{L_i} \! =0$, $f_1$ vanishes on $L_i\cap\{x_1\!\ne\!0\}$, which is dense in $L_i$, and so $f_1$ vanishes on $L_i$. Thus $f_1$ vanishes on the union $E$ of all $L_i$ that do not have the variable $x_1$. This excludes $L_1$, hence $E$ is a union of less than $s$ coordinate linear varieties, and we can apply the induction hypothesis \em on the number of coordinate varieties \em to the $\sS^{\nu-1}$ function $f_1$ to get an expression
$$
f_1(x)={\sum}_\tau f_{1\tau} x^\tau,
$$
for some $\sS^{\nu-1-\#(x^\tau)}$ functions $f_{1\tau}$. Here we have $\nu-1\ge\#(x^\tau)$. To check that, notice that in this sum, each monomial $x^\tau$ contains one variable from each $L_i$ in $E$ and all the others $L_i$ have the variable $x_1$, so that $x^\tau x_1$ is one of our generators of $I(X)$. Since $\#(x^\tau x_1)=1+\#(x^\tau)$, we have
$$
\nu-1\ge\max_\sigma\#(x^\sigma)-1\ge\#(x^\tau).
$$
Next we turn to the other summand $g(x')$. If some $L_i$ is the hyperplane $x_1=0$, then $f$ vanishes on that hyperplane, and so
$$
0=f(0,x')=f_1(x)0+g(x').
$$
We conclude 
$$
f(x)=f_1(x)x_1={\sum}_\tau f_{1\tau} x^\tau x_1
$$
and the argument is complete. Hence, suppose that all $L_i$ have some variable other than $x_1$.

We consider $x'$ as coordinates in $\R^{m-1}$ and denote $L'_i\subset\R^{m-1}$ the coordinate linear variety with the same equations that $L_i$, $x_1$ excluded in case it is in $L_i$, that is, $L_i'$ corresponds to $L_i\cap(\{0\}\times\R^{m-1})$ after erasing the first coordinate, which is $0$. Now, $g(x')$, as a function on $\R^{m-1}$, vanishes on $L'_1\cup\cdots\cup L'_s$. Indeed, if $z'\in L'_i$, then $z=(0,z')\in L_i$ and since $f$ vanishes on $L_i$,
$$
0=f(z)=f(0,z')=f_1(z)0+g(z'),
$$
as claimed. 

Thus, we can apply to $g(x')$, which is $\sS^\nu$, the induction hypothesis \em on the dimension \em to get:
$$
g(x')=\sum_\sigma g_\sigma(x'){x'}^\sigma,
$$
where the $g_\sigma$'s are $\sS^{\nu-\#(x^\sigma)}$. Here each monomial ${x'}^\sigma$ has one variable from each $L'_i$, hence from each $L_i$, and so it is in fact one generator of $I(X)$. 

All in all the expression
$$
f(x)=f_1(x)x_1+g(x')=\sum_\tau f_{1\tau}x^\tau x_1+\sum_\sigma g_\sigma{x'}^\sigma
$$
verifies all conditions required, and the proof is complete.
\end{proof}

As remarked before, this settles approximation for Nash sets with monomial singularities, with some restriction on differentiability classes. We state that solution in full below.

\begin{prop}\label{closext}
Let $X$ be a Nash set with monomial singularities in a Nash manifold $M\subset \R^a$ of dimension $m$. Let $\nu\!\ge\! m$ and
let $F:M\to\R$ be an $\sS^\nu$ function. Then every Nash function $h:X\to\R$ which is $\sS^\nu$ close enough to $f=F|_{X}$ has a Nash extension $H:M\to\R$ which is $\sS^{\nu-m}$ close to $F$. In particular, an $\sS^\nu$ function $F$ whose restriction to $X$ is Nash can be $\sS^{\nu-m}$ approximated by Nash functions $H$ that coincide with $F$ on $X$.
\end{prop}
\begin{proof}
First we prove the particular case when $F|_{X}$ is Nash. By Theorem \ref{thmfiniteness}, $X$ can be covered by finitely many open semialgebraic sets $U_i$ each equipped with a Nash diffeomorphism $u_i:U_i\rightarrow \R^m$ that maps $X\cap U_i$ onto a union of coordinate linear varieties. Hence by Lemmas \ref{lem:equivzeroapprox} and \ref{locpi} it is enough to prove $I^\nu(X\cap U_i)\subset I(X\cap U_i)\sS^{\nu-m}(U_i)$ for all $i$, but this follows from Proposition \ref{genidsol}. Now we deduce the general case. Let $\mathcal U\subset\sS^{\nu-m}(M)$ be an open $\sS^{\nu-m}$ neighborhood of $F$. Then $\mathcal U'=\mathcal U\cap\sS^\nu(M)$ is open in $\sS^\nu(M)$, and since the restriction $\rho:\sS^\nu(M)\to\sS^\nu(X)$ is an open homomorphism, $\rho(\mathcal U')$ is an open $\sS^\nu$ neighborhood of $f$ in $\sS^\nu(X)$. Thus, if our Nash function $h$ is in $\rho(\mathcal U')$, it has a semialgebraic $\sS^\nu$ extension $G\in\mathcal U'\subset\mathcal U$. By the particular case, there are Nash 
functions $H:M\to\R$ arbitrarily $\sS^{\nu-m}$ close to $G$ with $H|_{X}=h$. Since $\mathcal U$ is $\sS^{\nu-m}$ open, we can choose $H\in\mathcal U$.
\end{proof}

\section{Approximation for maps}\label{sec:approxmap}

The aim of this section is to prove approximation for maps instead of functions (Theorem \ref{approxgc}). In \ref{smapp} we defined and equipped with a topology the spaces of $\sS^\nu$ and ${}^{\tt c}\sS^\nu$ maps from a Nash set into a semialgebraic set. We apply extension for functions (Proposition \ref{topirre}) to prove that $\sS^\nu$ and ${}^{\tt c}\sS^\nu$ maps coincide for Nash sets with monomial singularities.

\begin{prop}\label{cmap}
Let $X\subset M$ be a Nash set with monomial singularities whose irreducible components are $X_1,\dots, X_s$ and let $T\subset \R^b$ be a semialgebraic set. The multiple restriction homomorphism $\sS^\nu(X,T)\to\sS^\nu(X_1,T)\times\cdots\times\sS^\nu(X_s,T)$ is a closed embedding that provides a topological identification $\sS^\nu(X,T)\equiv{}^{\tt c}\sS^\nu(X,T)$. 
\end{prop}
\begin{proof}
Indeed, since ${}^{\tt c}\sS^\nu(X,\R^b)={}^{\tt c}\sS^\nu(X,\R)^b$, Proposition \ref{topirre} gives the result for $T=\R^b$. Then, for arbitrary $T\subset\R^b$, we have the commutative diagram
\setlength{\unitlength}{1mm}
\begin{center}
\begin{picture}(50,16)(0,-1)
\put(25,12){\makebox(0,0){$\,\sS^\nu(X,T)\,\,\longrightarrow\,\,{}^{\tt c}\sS^\nu(X,T)\,\hookrightarrow\,\sS^\nu(X_1,T)\,\times\cdots\times\,\sS^\nu(X_s,T)$}}
\put(-14,8.5){\vector(0,-1){4.5}}\put(34.5,8.5){\vector(0,-1){4.5}}\put(66,8.5){\vector(0,-1){4.5}}
\put(26,0){\makebox(0,0){$\sS^\nu(X,\R^b)\longrightarrow{}^{\tt c}\sS^\nu(X,\R^b)\hookrightarrow\sS^\nu(X_1,\R^b)\times\cdots\times\sS^\nu(X_s,\R^b)$}}
\end{picture}
\end{center}
As the result is true for the lower row and the vertical arrows are continuous, it follows for the upper one. 
\end{proof}

Relative approximation extends straightforwardly to maps into $\R^b$. Moreover, we can use Nash tubular neighborhoods to show the following.

\begin{prop}\label{closextmap}
Let $M\subset \R^a$ and $N\subset \R^b$ be Nash manifolds of dimensions $m$ and $n$ respectively and let $X\subset M$ be a Nash set with monomial singularities. Let $\nu\!\ge\! m$ and let $F:M\to N$ be an $\sS^\nu$ map. Then every Nash map $h:X\to N$ which is $\sS^\nu$ close enough to $f=F|_{X}$ has a Nash extension $H:M\to N$ which is $\sS^{\nu-m}$ close to $F$.
\end{prop}
\begin{proof}
Consider a Nash retraction $\eta:W\to N$ of $N$ in $\R^b$ (see Fact \ref{fact:retract}). 
By Proposition \ref{closext} 
there is a Nash extension $G:M\to\R^b$ of $h$ which is $\sS^{\nu-m}$ close to $F$ as a map into 
$\R^b$. Since $F(X)\subset N\subset W$ we can choose the approximation close enough so that $G(X)\subset W$, and then 
$H=\eta\circ G:M\to N$ is a well defined Nash map that extends $h$. But composition on the left is continuous (Proposition \ref{comr}) and therefore the composite $H=\eta\circ G$ is close to the composite $\eta\circ F=F$, provided $G$ is close enough to $F$.
\end{proof}

The preceding extension is the key fact to obtain an approximation result for maps into Nash sets with monomial singularities. What we want is to approximate differentiable semialgebraic maps $f:X\to Y$ of Nash sets. The first restriction is that $f$ must \em preserve irreducible components, \em that is, it must 
map each irreducible component $X_i$ of $X$ into one $Y_k$ of $Y$. Indeed, suppose that for each $k$ the component $X_i$ 
contains a point $x_k$ with $f(x_k)\notin Y_k$ and pick $\veps>0$ smaller than all distances 
$\dist(f(x_k), Y_k)$. If $f$ can be approximated by Nash maps, pick one $g:X\to Y$ that is $\veps$ 
close. Since $g$ is Nash, the inverse images $g^{-1}(Y_k)$ are Nash sets and therefore, since the irreducible $X_i$ is contained in their union, we deduce that $g(X_i)\subset Y_k$ for some $k$. 
Thus $g(x_k)\in Y_k$ and we get
$$
 \dist(f(x_k),Y_k)\le\dist(f(x_k),g(x_k))<\veps,
$$
a contradiction. 

To progress further we pause to look at Nash sets with monomial singularities from a global viewpoint. We recall that among \em local \em normal crossings one distinguishes the so-called  \em normal crossing 
divisors \em by asking their irreducible components to be non-singular \cite[1.8]{fgr}. This is a global condition that can be applied to monomial singularities. We call a Nash set with monomial singularities $X$ a \em Nash monomial crossings \em if its irreducible components $X_1,\dots,X_s$ are all non-singular; in other words, $X_1,\dots,X_s$ are Nash manifolds. 
\begin{prop}\label{glgc}
Let $X$ be a Nash monomial crossings whose irreducible components we denote $X_1,\dots,X_s$. Then all intersections $X_{i_1}\cap\cdots\cap X_{i_r}$ are Nash manifolds, and any union of them is a Nash monomial crossings.
\end{prop}
\begin{proof}
By Lemma \ref{lem:irreducohe} it is enough to show that every intersection $X_I=X_{i_1}\cap\cdots\cap X_{i_r}$ is an affine Nash manifold. Fix a point $x\in X_I$. The distinct irreducible components of the germ $X_x$ are the germs $X_{i,x}$ such that $x\in X_i$. Indeed, if, say, $X_{i,x}\subset X_{j,x}$, then $\dim(X_i)=\dim(X_i\cap X_j)$ and since $X_i$ is irreducible, $X_i\subset X_j$, a contradiction. Now pick any Nash coordinates $u$ of $M$ at $x$ such that $u(x)=0$ and $X_x={\bigcup}_{\lambda\in\varLambda}\{u_{\lambda}\!=0\}_x$. Observe that $u$ maps each irreducible component $X_{i,x}$ of $X_x$ onto a coordinate linear variety $L_i\subset\R^m$. In particular this happens for $i=i_1,\dots,i_r$, and we deduce 
$$
u(X_{I,x})=u\big(X_{i_1,x}\cap\cdots\cap X_{i_r,x}\big)=L_{i_1,0}\cap\cdots\cap L_{i_r,0}.
$$
Since any intersection of linear varieties is a linear variety, the germ $X_{I,x}$ is non-singular, and we are done.
\end{proof}

After this remark we can prove approximation for Nash maps between Nash set with monomial singularities. 

\begin{stepsp}{\em Proof of Theorem \em \ref{approxgc}.}\label{pfapproxgc} By hypothesis, $f$ maps each irreducible component $X_i$ of $X$ into some irreducible component of $Y$; we choose one and denote it by $Y_{k(i)}$. As usual, we classify the intersections $Y_i$'s into levels: at level $p$ we have the intersections $Y_I=\bigcap_{i\in I}Y_i$ with $\#I=p$. At bottom level $s$ we have the intersection $Y_1\cap\cdots\cap Y_s$ of all irreducible components. In general, at level $p$ we have the union $Y^{(p)}=\bigcup_{\#I=p}Y_I$ of intersections $Y_I$ at level $p$. Now, we denote $X_I$ the intersection of all $X_i$'s with $k(i)\in I$, and say this is an intersection of $X_i$'s at level $p$. Quite naturally, we denote $X^{(p)}$ the union of all these $X_I$. We keep in mind this leveled collection of $X_I$'s as an inverse image of the leveled collection of $Y_I$'s.

Observe that in this construction there may be empty intersections. Thus, we stop at the last non-empty $X^{(r)}=\bigcup_IX_I$, which implies that all $X_I$'s at this level are disjoint. 

After this organization of data, note that every $X^{(p)}$ is a Nash set with monomial singularities and the corresponding $Y^{(p)}$ is a Nash monomial crossings (Lemma \ref{lem:irreducohe} and Proposition \ref{glgc}). To complete the setting, note that $f$ restricts to an $\sS^\nu$ map $f^{(p)}:X^{(p)}\to Y^{(p)}$. We are to approximate every $f^{(p)}$ starting at the bottom level $p=r$ and climbing to level $p=1$. In each jump of level we will use Proposition \ref{closextmap} and therefore there will be a loss of differentiability. As $f=f^{(1)}$ this will complete the proof. The proof runs in several steps.

\em Step 1. \em We start at level $p=r$. Since $X^{(r)}\ne\varnothing$, we have the map $f^{(r)}:X^{(r)}\to Y^{(r)}$ and also $Y^{(r)}\ne\varnothing$. Since the union $X^{(r)}=\bigcup_IX_I$ is disjoint and $f^{(r)}$ maps $X_I$ into the corresponding $Y_I$ we only have to approximate the restriction $f^{(r)}|_{X_I}:X_I\to Y_I$ for $X_I\ne\varnothing$ (and so also $Y_I\ne\varnothing$). But $Y_I$ is one single intersection, hence an affine Nash manifold (Proposition \ref{glgc}). By Proposition \ref{aproxM} there are Nash maps $g_I:X_I\to Y_I$ arbitrarily $\sS^\nu$ close to $f^{(r)}|_{X_I}$. These $g_I$'s glue into the desired approximation $g^{(r)}$ of $f^{(r)}$. The construction itself guarantees that $g^{(r)}$ maps each intersection $X_I$ at level $r$ into the corresponding $Y_I$.

\em Step 2. \em Now suppose we are at level $p\ge r$ (hence $X^{(p)}\ne\varnothing$), and we have Nash maps $g^{(p)}:X^{(p)}\to Y^{(p)}$ arbitrarily $\sS^{\nu-mj}$ close to $f^{(p)}:X^{(p)}\to Y^{(p)}$ where $j(=r-p)$ is the number of levels already done. Moreover, we have that both $f^{(p)}$ and $g^{(p)}$ map each non-empty intersection $X_I$ at level $p$ into $Y_I$. The argument that follows only needs these non-empty intersections. In view of our discussion of topologies and irreducible 
components (Proposition \ref{topirre}), the restrictions $g_I=g^{(p)}|_{X_I}:X_I\to Y_I$ ($I$ with $p$ indices 
and $X_I\ne\varnothing$) are Nash approximations of the restrictions $f_I=f^{(p)}|_{X_I}:X_I\to Y_I$. Let us now approximate $f^{(p-1)}:X^{(p-1)}\to Y^{(p-1)}$.

\em Step 3. \em Let $X_J$ be an intersection of $X_i$'s at level $p-1$. Every intersection of $X_J$ with another $X_{J'}$ at 
the same level $p-1$ is included in some intersection $X_I$ at level $p$, hence consider the restrictions to 
$X_J\cap X^{(p)}$ of $f^{(p)}$ and of $g^{(p)}$; let $\varphi$ stand for the first and $\psi$ for the second 
one. Of course, we only care for the non-empty intersections of that type.
The intersection $X_J\cap X^{(p)}$ is a Nash set with monomial singularities, and the restrictions are maps into $Y_J$, which is a Nash 
manifold because $Y$ is a Nash monomial crossings (Proposition \ref{glgc} again). By Proposition \ref{closextmap} if $\psi$ is $\sS^\mu$ close enough to $\varphi$, it has an extension $\psi_J:X_J\to Y_J$ 
arbitrarily $\sS^{\mu-m}$ close to $f_J$ (in fact, that proposition gives an extension to $M$ that we restrict 
to $X_J$). In our case we take
$\mu=\nu-mj$, hence $\mu-m=\nu-m(j+1)$. The $\psi_J$'s glue together into a $\tt c$-Nash 
map $g^{(p-1)}:X^{(p-1)}\to Y^{(p-1)}$ arbitrarily $\sS^{\nu-m(j+1)}$ close to $f^{(p-1)}:X^{(p-1)}\to Y^{(p-1)}$. Notice that $g^{(p-1)}$ maps each non-empty intersection $X_J$ at level $p-1$ into $Y_J$. By Proposition \ref{cmap} the $\tt c$-Nash maps are Nash, also concerning topologies of 
maps, and we have 
obtained the desired approximation of $f^{(p-1)}$. 

\em Final arrangement. \em In the process above the differential class loses $m$ units at every level jump. As we start at level $r$ with an $\sS^\nu$ approximation, in the end have an $\sS^{\nu-m(r-1)}$ approximation. Now it remains to prove that $r\le \binom{n}{[n/2]}$. Thus pick any point $x\in X_I$ with $X_I\ne\varnothing$ at level $r$. Since $f(X_I)\subset Y_I$, the point $y=f(x)$ is in $Y_I=Y_{k_1}\cap\cdots\cap Y_{k_r}$, and the $r$ germs $Y_{k_1,y},\dots,Y_{k_r,y}$ are irreducible components of $Y_y$. Since $Y\subset N$ has a monomial singularity at $y$, its tangent cone is an extremal family, and by Remark \ref{sperner} it has at most $\binom{n}{[n/2]}$ elements. This gives the required bound for $r$.
\fin
\end{stepsp}
\begin{remark}\label{rmk:qglobalnormal}Note that the bound $\binom{n}{[n/2]}$ is sharp: we could have all the coordinates varieties of dimension $[n/2]$. The loss of differentiability class in the approximation obtained above has been formulated to show that it only depends on dimensions. A different obvious bound for $r$ is the number of irreducible components $Y_{k(i)}$ of $Y$ we had chosen to start with, and that number can certainly be smaller than $\binom{n}{[n/2]}$. For instance, if $Y\subset N$ is a normal crossing divisor, i.e., $\text{codim}(Y)=1$, then $r\leq n$ and so Theorem \ref{approxgc} holds for $q=m(n-1)$. However, in general it is difficult to get a better estimation of its size if we do not know a priori the dimensions of the $Y_i$'s. \fin
\end{remark}

\section{Classification of affine Nash manifolds with corners}\label{sec:corners}

In this section we use our approximation results to deduce Theorem \ref{difcor} concerning the classification of affine Nash manifolds with corners. First, recall that a map $h:Z\to T$ is an $\sS^\nu$ diffeomorphism if it is a bijection and both $h$ and $h^{-1}$ are $\sS^\nu$ maps. With respect to approximation, diffeomorphisms of affine Nash manifolds behave well.

\begin{fact}\label{diffo}\cite[II.1.7, p.\hspace{1.5pt}86]{s} Let $h:M\to N$ be an $\sS^\nu$ diffeomorphism of affine Nash manifolds. If an $\sS^\nu$ map $g:M\to N$ is $\sS^\nu$ close enough to $h$, then $g$ is also an $\sS^\nu$ diffeomorphism, and $g^{-1}$ is $\sS^\nu$ close to $h^{-1}$.
\end{fact}

In particular, from this and Proposition 2.D.3 we deduce that for all $\nu\ge 1$ \em every $\sS^\nu$ diffeomorphism $f:M\to N$ can be approximated by Nash diffeomorphisms, \em hence the $\sS^\nu$ classification and the Nash classification coincide for affine Nash manifolds. Our Theorem 1.8 says this is true for manifolds \em with corners \em (for suitable $\nu$). 

For the proof of Theorem \ref{difcor} we need some results that can be of interest by themselves. 

\begin{lem}\label{exthomeo}
Let $f:T\to T'$ be a  semialgebraic local homeomorphism of locally compact semialgebraic sets which restricts to a homeomorphism from a closed semialgebraic subset $S$ of $T$ onto another $S'$ of $T'$. Then $f$ restricts to a homeomorphism from an open semialgebraic neighborhood $W$ of $S$ in $T$ onto another $W'$ of $S'$ in $T'$.
\end{lem}
\begin{proof}We are to prove first that we may assume $f^{-1}(S')=S$. Consider the semialgebraic set
$$
C=\{x\in T\setminus S: \text{there is some $y\in S$ such that $f(y)=f(x)$}\}.
$$
We claim that no point in $S$ is adherent to $C$. For, suppose there are sequences $\{x_k\}_k$ off $S$ converging to $x\in S$ and $\{y_k\}_k$ in $S$ such that $f(y_k)=f(x_k)$. Then
$$
{\lim}_{k} f(y_k)={\lim}_{k} f(x_k)=f(x)\in S'
$$
and since $f|_S:S\to S'$ is a homeomorphism, $\{y_k\}_k$ must converge to $x$. But $f$ is injective near $x$, hence $x_k=y_k\in S$ for $k$ large, a contradiction. Thus replacing $T$ by $T\setminus \overline{C}$ and $T'$ by its open image (since $f$ is a local homeomorphism it is an open map)  we can assume that $f^{-1}(S')=S$.

Now, since $f$ is a local homeomorphism, there is a finite open semialgebraic covering $U_1,\dots,U_r$ of $T$ such that all restrictions $f|_{U_i}:U_i\to f(U_i)$ are homeomorphisms \cite[9.3.9]{bcr}. Since $f$ is open, we can also suppose that the $f(U_i)$'s form a cover of $T'$. If $r=1$ we are done, so let us see that when $r>1$ we can modify the covering to reduce the number $r$ of open sets, which will give the result.

First of all we consider the open semialgebraic set $\varDelta=f(U_1)\cup f(U_2)$ and find a semialgebraic shrinking of the covering $\{f(U_1),f(U_2)\}$, i.e., open semialgebraic sets 
$\varDelta_1,\varDelta_2$ that still cover $\varDelta$ and such that $\varDelta\cap \overline{\varDelta}_i\subset f(U_i)$ (for instance, take a continuous semialgebraic function $\delta$ on $\varDelta$ with $\delta|_{\varDelta\setminus f(U_1)}=-1$ and $\delta|_{\varDelta\setminus f(U_2)}=1$ using \cite[2.6.9]{bcr} and set $\varDelta_1=\{\delta>-\frac{1}{2} \}$ and $\varDelta_2=\{ \delta<\frac{1}{2}\}$).

Next, consider the open semialgebraic sets 
$V_i=U_i\cap f^{-1}(\varDelta_i)$. We claim that
\begin{equation}\label{eq2}
S\cap(V_1\cup V_2)=S\cap (U_1\cup U_2).
\end{equation}

Indeed, for the relevant inclusion right to left, let us consider $x\in S\cap(U_1\cup U_2)$. Then $
f(x)\in S'\cap (f(U_1)\cup f(U_2))=S'\cap\varDelta$. We can assume that $f(x)\in\varDelta_1\subset f(U_1)$, so that $f(x)=f(y)$ for some $y\in U_1$. Thus $y\in f^{-1}(S')=S$ and, since $f$ is injective on $S$, we deduce that $x=y\in U_1\cap f^{-1}(\varDelta_1)=V_1$.

Now we show that no $x\in S\cap(U_1\cup U_2)$ is adherent to the semialgebraic set
$$
C'=\{x\in V_1\setminus V_2: \text{there is some $y\in V_2$ such that $f(y)=f(x)$}\}.
$$
Indeed, suppose there are sequences $\{x_k\}_k$ in $V_1\setminus V_2$ converging to $x\in S\cap(U_1\cup U_2)$ and $\{y_k\}_k$ in $V_2$ with $f(y_k)=f(x_k)$. Since $f(x)\in f(U_1)\cup f(U_2)=\varDelta$ we have 
$$
f(x)={\lim}_{k} f(x_k)={\lim}_{k} f(y_k)\in\varDelta\cap \overline{f(V_2)}\subset\varDelta\cap\overline{\varDelta_2}\subset f(U_2),
$$
and since $f|_{U_2}:U_2\to f(U_2)$ is a homeomorphism, there exists  $y\in U_2$ with $y=\lim_{k} y_k$. Consequently, $f(y)=f(x)\in S'$. As $f^{-1}(S')=S$, we have $y\in S$ and since $f|_{S}$ is injective, $y=x$. Thus the two sequences $\{x_k\}_k$ and $\{y_k\}_k$ converge to $x$ and $f$ being locally injective, $x_k=y_k$ for $k$ large, a contradiction. We have so proved that 
\begin{equation}\label{eq3}
S\cap(U_1\cup U_2)\cap\overline{C'}=\varnothing.
\end{equation}

Now set $W_2=(V_1\cup V_2)\setminus \overline{C'}$, which is an open semialgebraic neighborhood of 
$$
S\cap W_2=S\cap(V_1\cup V_2)=S\cap (U_1\cup U_2)
$$
the first equality by equation \eqref{eq3} and the second one by equation \eqref{eq2}.

Finally, by definition of $C'$, $f$ is injective on $W_2$, and so $f$ restricts to a homeomorphism $W_2\to f(W_2)$. All in all we can replace the two open sets $U_1$ and $U_2$ of the initial covering $U_1,\dots,U_r$ by the single open set $W_2$ so getting a new covering of $S$ by $r-1$ open sets, as desired.
\end{proof}

\begin{lem}\label{indextdpm}
Let $Q\subset\R^a$ be a Nash manifold with corners. Let $M\subset \R^a$ be a Nash envelope of $Q$ such that the Nash closure $X$ of $\partial Q$ in $M$ is a normal crossings. Let $N\subset \R^b$ be a Nash manifold and $f:Q\to N$, $g:X\to N$ be $\sS^\nu$ maps such that $g|_{\partial Q}=f|_{\partial Q}$. Then for $M$ small enough there exists an $\sS^\nu$ map $H:M\to N$ such that $H|_Q=f$ and $H|_{X}=g$.
\end{lem}
\begin{proof}Let $m=\dim(M)=\dim(Q)$. By Theorem \cite[Thm.\hspace{2pt}1.11]{fgr} there is a finite open covering $U_1,\ldots,U_{\ell}$ of $\partial Q$ equipped with Nash diffeomorphisms $(u_{i1},\ldots,u_{im}):U_i\to \R^m$ with $U_i\cap X=\{u_{i1}\cdots u_{ik_i}=0 \}$ and $U_i\cap Q=\{u_{i1}\geq0,\ldots,u_{ik_i}\geq 0\}$ for some $1\leq k_i\leq m$. By Corollary \ref{cor:extcorners} for each $U_i$ there is an $\sS^\nu$ map $\zeta_i:U_i\to\R^b$ such that $\zeta_i=f$ on $U_i\cap Q$ and $\zeta_i=g$ on $U_i\cap X$.  Set $U_0=\text{Int}(Q)$ and $\zeta_0=f|_{U_0}$. Let $\sigma_0, \sigma_1,\ldots, \sigma_{\ell}$ be an $\sS^\nu$ partition of unity subordinated to $\{U_0, U_1,\ldots,U_{\ell}\}$. Consider $U=\bigcup_{i=0}^\ell U_i$ and the $\sS^\nu$ map $H=\sum_{i=0}^{\ell}\sigma_i\zeta_i:U\to \R^b$ which coincides with $f$ on $Q$ and with $g$ on $X$. This map $H$ solves the problem except that its image is not contained in $N$. To amend this, pick a Nash retraction $\rho:V\to N$ of $N$ in $\R^b$ and replace $M$ by $H^{-1}(V)$ and $H$ by $\rho \circ H$. 
\end{proof}

\begin{prop}\label{extdpm}
Let $Q\subset \R^a$ be a Nash manifold with divisorial corners. Let $M\subset \R^a$ be a Nash envelope of $Q$ such that the Nash closure $X$ of $\partial Q$ in $M$ is a Nash normal crossing divisor. Let $X_1,\ldots,X_s$ be the irreducible components of $X$ and let $Y$ be a Nash monomial crossings in a Nash affine manifold $N\subset \R^b$. Let $f:Q\to N$ be an $\sS^\nu$ map such that each $f(\partial Q\cap X_i)$ is contained in an irreducible component $Y_{k(i)}$ of $Y$. Then for $M$ small enough there exists an $\sS^\nu$ extension $F:M\to N$ of $f$ such that $F(X_i)\subset Y_{k(i)}$ for $i=1,\ldots,s$.
\end{prop}
\begin{proof}We will use the notion of iterated faces introduced in \ref{facesNashmancorners}. Recall that $Q$ is the unique $m$-face and for $d<m$ a $d$-face is a face of a $(d+1)$-face. This construction ends at some $d=m_0\ge 0$, so that the $m_0$-faces are affine Nash manifolds. For $M$ small enough we can assume the properties of Proposition \ref{facesoffaces}.

Now, to prove the statement it is enough to construct, for each Nash closure $Z$ of an iterated face, an $\sS^\nu$ map $F^Z: Z\to N$ such that

(i) $F^Z|_{Z\cap Q}=f|_{Z\cap Q}$,

\hspace{-1mm}(ii) if $Z'$ is the Nash closure of an iterated face with $Z'\subset Z$ then $F^{Z'}=F^{Z}|_{Z'}$,

\hspace{-2mm}(iii) if $Z\subset X_i$ then $\text{im}(F^Z)\subset Y_{k(i)}$.

We proceed by ascending induction on the dimension $d$ of the faces. For $d=m_0$,  every $m_0$-face is an affine Nash manifold contained in $Q$ and therefore it coincides with its Nash closure, so that we can just take the restriction of $f$ to that $m_0$-face. Now, fix the Nash closure $Z$ of a $d$-face $D$, $d>m_0$, for which we know $Z\cap Q=D$. Let $Z_1,\ldots,Z_\ell$ be the irreducible components of the Nash closure of $\partial D$ in $Z$.  
Note that each $Z_j$ is the Nash closure of a face $D_j$ of $D$, which is a $(d-1)$-face. By induction, properties (i)-(iii) hold for the $Z_j$'s. Let $I=\{i: Z\subset X_i\}$ and consider $Y_I=\bigcap_{i\in I}Y_{k(i)}$, which is an affine Nash manifold by Proposition \ref{glgc}. 

Let $F^{Z_j}:Z_j\to Y_I$ be the $\sS^\nu$ maps provided by the induction hypothesis for $j=1,\ldots, \ell$, which are well defined by property (iii). By Proposition \ref{facesoffaces}(3), the connected components of $Z_j\cap Z_{j'}$  meeting $Q$ are Nash closures of iterated faces (of smaller dimensions) and therefore, by (ii),  $F^{Z_j}$ and $F^{Z_{j'}}$ coincide on those connected components. Let $C$ be the union of the connected components of all intersections $Z_j\cap Z_{j'}$ that do not intersect $Q$. Since $C$ is a closed semialgebraic set disjoint from $Q$, we can shrink $M$ to $M\setminus C$. In particular, we can assume that the map
$$
G:\bigcup_{j=1}^\ell Z_j \to Y_I \quad \text{where } G=F_{Z_j} \text{ on } Z_j
$$
is a well-defined ${}^c\sS^\nu$ function. But that union $\bigcup_j Z_j$ is a normal crossings in $Z$, hence $G$ is $\sS^\nu$ by Proposition \ref{cmap}.

Finally, by (i), $G|_{\partial D}=f|_{\partial D}$ and we can apply  Lemma \ref{indextdpm} to find an $\sS^\nu$ map $F^Z:Z\to Y_I$ such that $F^Z|_{D}=f|_{D}$ and $F^Z|_{\partial D}=G|_{\partial D}$.  This as usual after shrinking $Z$ to $Z\setminus C$ for some closed semialgebraic set $C\subset Z$, that is, after shrinking $M$ to $M\setminus C$.
\end{proof}

\begin{stepsp}{\em Proof of Theorem \em\ref{difcor}.}\label{pfdifcor} 
We can suppose $Q_1$ and $Q_2$, hence their interiors, connected. Let $h:Q_1\to Q_2$ be an $\sS^\nu$ diffeomorphism. Let $M\subset \R^a$ and $N\subset \R^b$ be Nash envelopes of $Q_1$ and $Q_2$. Let $X$ and $Y$ be the Nash closures of $\partial Q_1$ and $\partial Q_2$ in $M$ and $N$. By hypothesis both $X$ and $Y$ are Nash normal crossing divisors. Now, since $h$ is an $\sS^\nu$ diffeomorphism, we show that \em it maps the intersection of $\partial Q$ with each irreducible component of $X$ into some irreducible component of $Y$\em. For, let $X'$ be an irreducible component of $\partial X$, that is, the Nash closure of a face $D$ of $\partial Q_1$. By definition, $D$ is the topological closure of a connected component $C$ of ${\tt Smooth}(\partial Q_1)$ and we can assume that $X'\cap \partial Q_1=D$. Clearly, $h(C)$ is an $\sS^\nu$ manifold open in $\partial Q_2$ and therefore $h(C)\subset {\tt Smooth}(\partial Q_2)$ (just note that in $\partial Q_2$ being a smooth point in the Nash or $\sS^\nu$ sense coincide). Then, $h(C)$ is contained in a connected component of ${\tt Smooth}(\partial Q_2)$ and therefore $h(C)\subset Y'$ where $Y'$ is an irreducible component of $Y$. In particular, $h(\partial Q_1\cap X')=h(D)=h(\overline{C})=\overline{h(C)}\subset \overline{Y'}=Y'$, as required.

Now, by Proposition \ref{extdpm} we can assume there is an $\sS^\nu$ extension $H:M\to N$ of $h$ such that $H(X)\subset Y$. This extension is a local diffeomorphism at every point of $Q_1$, hence shrinking $M$ and $N$ we may assume that $H$ is a local diffeomorphism. But it is a homeomorphism from $Q_1$ onto $Q_2$ and hence by Lemma \ref{exthomeo} after a new shrinking $H$ is a diffeomorphism. 

After this preparation, we use Theorem \ref{approxgc} and Remark \ref{rmk:qglobalnormal} to obtain a semialgebraic $\sS^{\nu-q}$ approximation of $g=H|_X:X\to Y$ by a Nash map $f:X\to Y$, where $q=m(m-1)$. Then, by Proposition \ref{closextmap}, $f$ has a Nash extension $F:M\to N$ which is $\sS^{\nu-q-m}$ close to $H$. Note that 
$$
\nu-q-m=\nu-m(m-1)-m=\nu-m^2\ge1.
$$
Since $H$ is a diffeomorphism, a close enough approximation $F$ is also a diffeomorphism (Fact \ref{diffo}). In 
particular, $X$ is Nash equivalent to $Y$.

Let us check that $F(Q_1)=Q_2$. Since the manifolds are the closures of their interiors, it is enough to see that $F$ maps $\Int(Q_1)$ onto $\Int(Q_2)$. We know that $F(X)=f(X)=Y$, and consequently 
$F(M\setminus X)=N\setminus Y$. Since $Q_1\cap X=\partial Q_1$, the interior $\Int(Q_1)$ is a connected component of $M\setminus X$, hence $F$ maps it onto one of $N\setminus Y$. Among these latter is $\Int(Q_2)$ (also $Q_2\cap Y=\partial Q_2$), and we know that $H(\Int(Q_1))=\Int(Q_2)$. Pick any point $x\in\Int(Q_1)$. Since $\Int(Q_2)$ is open in $N$, we have
$\dist(H(x),N\setminus\Int(R))>0$ and for $F$ close to $H$ we conclude $F(x)\notin N\setminus\Int(Q_2)$. Consequently $F(\Int(Q_1))\cap\Int(Q_2)\ne\varnothing$ and so $F(\Int(Q_1))=\Int(Q_2)$. 
\fin
\end{stepsp}
\begin{remark}Let us review the preceding proof for two Nash manifolds $Q_1,Q_2$ with boundary \em without corners\em. Notice that their boundaries $\partial Q_1, \partial Q_2$ are Nash smooth hypersurfaces of the Nash envelopes.

First of all,  the boundaries are their own Nash closures, that is, $X=\partial Q_1$ and  $Y=\partial Q_2$. Hence, Proposition \ref{extdpm} is not needed.  On the other hand, Theorem \ref{approxgc} can be replaced by Fact \ref{aproxM} and this approximation does not lower the differentiability class. Thus, we have a Nash map $f:X\to Y$ which is an $\sS^{\nu}$ approximation of $g=H|_X$.

Next, the Nash extension $F$ of $f$ provided by Proposition \ref{closextmap} is an $\sS^{\nu-1}$ approximation of $H$. Indeed, this proposition comes from:

(i) Lemma \ref{locpi} and Proposition \ref{genidsol}, here needed for one coordinate hyperplane only, and therefore giving $I^\nu(X)\subset I(X)\sS^{\nu-1}(M)$, and

\hspace{-1mm}(ii) Proposition \ref{closext}, which for a Nash smooth hypersurface $X$ of $M$ gives by (i) a Nash extension that is an $\sS^{\nu-1}$ approximation.

This leads to an $\sS^{\nu-1}$ approximation $F$ of $H$, as claimed.

All in all, we conclude: \em Two $m$-dimensional affine Nash manifolds with boundary without corners which are $\sS^2$ diffeomorphic are Nash diffeomorphic.\em 
\end{remark}

\end{document}